\documentclass{article}

\usepackage[utf8]{inputenc}
\usepackage[T1]{fontenc}
\usepackage{graphicx}
\usepackage{tikz}
\usepackage{amsmath,amsfonts,bm,amsthm}
\usepackage{multirow}
\usepackage[pdftex,colorlinks=true,linkcolor=blue,citecolor=blue,urlcolor=blue]{hyperref}

\newcommand{\R}{\mathbb{R}}

\newcommand{\y}{\bm{y}}

\newtheorem{remark}{Remark}
\newtheorem{theorem}{Theorem}

\title{New time domain decomposition methods for parabolic optimal control problems II: Neumann-Neumann algorithms}

\author{Martin Jakob Gander$^{1}$, Liu-Di LU$^{1}$}
\date
{%
\noindent{\small\textit{$^1$Section of Mathematics, University of Geneva, Rue du Conseil Général 7-9, 1205 Geneva, Switzerland}}\\
}

\begin{document}

\maketitle

\begin{abstract} 
We present new Neumann-Neumann algorithms based on a time domain decomposition applied to unconstrained parabolic optimal control problems. 
	After a spatial semi-discretization, 
	the Lagrange multiplier approach provides a coupled forward-backward optimality system, 
	which can be solved using a time domain decomposition. 
	Due to the forward-backward structure of the optimality system, 
	nine variants can be found for the Neumann-Neumann algorithms. 
	We analyze their convergence behavior and determine the optimal relaxation parameter for each algorithm. 
	Our analysis reveals that the most natural algorithms are actually only good smoothers, 
	and there are better choices which lead to efficient solvers. 
	We illustrate our analysis with numerical experiments.
\end{abstract}

\begin{paragraph}{Keywords: }
Time domain decomposition, 
	Neumann-Neumann algorithm, 
	Parallel in Time, 
	Parabolic optimal control problems, 
	Convergence analysis.
\end{paragraph}

\begin{paragraph}{MSCcodes: }
65M12, 65M55, 65Y05,
\end{paragraph}

\section{Introduction}
As our model problem, 
we consider a parabolic optimal control problem: 
for a given target function $\hat y\in L^2(Q)$, $\gamma>0$ and $\nu\geq 0$, 
we want to minimize the cost functional 
\begin{equation}\label{eq:J}
  J(y,u) := \frac12 \|y -\hat y\|^2_{L^2(Q)}  + \frac{\gamma}2 \|y(T) - \hat y(T)\|^2_{L^2(\Omega)} + \frac{\nu}2\|u\|^2_{U_{\text{ad}}},
\end{equation}
subject to the linear parabolic state equation:
\begin{equation}\label{eq:heat}
    \begin{aligned}
      \partial_t y - \Delta y &= u \quad &&\text{ in } Q:=\Omega\times (0,T), \\
      y &=0 &&\text{ on }\Sigma:=\partial\Omega\times(0,T), \\
      y(0) &=y_0 &&\text{ on }\Sigma_0:=\Omega\times\{0\},
    \end{aligned}
\end{equation}
where $\Omega\subset\R^d$, 
$d=1,2,3$ is a bounded domain with boundary $\partial\Omega$, 
and $T$ is the fixed final time. 
The control $u$ on the right-hand side of the PDE is in an admissible set $U_{\text{ad}}$,
and we want to control the solution of the parabolic PDE~\eqref{eq:heat} toward a target state $\hat y$. 
For simplicity, 
we consider homogeneous boundary conditions.
The parabolic optimal control problem~\eqref{eq:J}-\eqref{eq:heat} leads to necessary first-order optimality conditions (see e.g.,~\cite{Lions1971,Troltzsch2010}),
which include a forward in time primal state equation~\eqref{eq:heat},
a backward in time dual state equation,
\begin{equation}
  \begin{aligned}
    \partial_t \lambda + \Delta \lambda &= y - \hat y \quad &&\text{ in } Q, \\
    \lambda &=0 &&\text{ on }\Sigma, \\
    \lambda(T) &=-\gamma(y(T)-\hat y(T)) &&\text{ on }\Sigma_T:=\Omega\times\{T\},
  \end{aligned}
\end{equation}
and an algebraic equation $\lambda = \nu u$ with $\lambda$ the dual state.
This forward-backward system cannot be solved by standard time-stepping methods, 
and has to be solved either iteratively or at once.
Solving at once the space-time discretized system can be challenging, 
especially for spatial dimension larger than one.
To overcome this challenge,
one can use gradient type methods by solving sequentially forward-backward systems~\cite{Hinze2009,Troltzsch2010}.
Multigrid methods~\cite{Abbeloos2011,Borzi2011,Hackbusch1981,Li2017358},
tensor product techniques~\cite{Bunger2020,Gunzburger2011,Kollmann2013469,Langer2016},
model order reduction~\cite{Alla2015,Iapichino2016,Kammann2013,Kunisch2004},
can also be applied to solve such problems.
Since the role of the time variable in forward-backward optimality systems
is key, 
it is natural to seek efficient solvers through Parallel-in-time techniques.
This includes,
waveform relaxation~\cite{Lelarasmee1982,Halpern2010},
Parareal~\cite{Lions2002},
PITA~\cite{Farhat2003},
PFASST~\cite{Emmett2012}, 
MGRIT~\cite{Falgout2014}, 
see also the survey paper~\cite{Gander2015}.
Application of such techniques to treat parabolic optimal control problems can be found in~\cite{Fang2022,Gander2020,Gotschel2019,Heinkenschloss2005}.

In~\cite{Gander2023},
we considered a new time domain decomposition approach motivated by \cite{Gander2016,Kwok2017},
and analyzed the convergence behavior of Dirichlet-Neumann and Neumann-Dirichlet algorithms within this framework.
We have surprisingly discovered different variants of Dirichlet-Neumann and Neumann-Dirichlet algorithms for the parabolic optimal control problem~\eqref{eq:J}-\eqref{eq:heat}, 
when decomposing in time.
This is mainly due to the forward-backward structure of the optimality system.
The present paper is the sequel of~\cite{Gander2023}:
our goal is to investigate Neumann-Neumann techniques~\cite{Bjorstad1986} in the context of time domain decomposition and analyze their convergence behavior.
We consider a semi-discretization in space and focus on the time variable.
This consists in replacing the spatial operator $-\Delta$ by a matrix $A\in\R^{n\times n}$, 
for instance using a Finite Difference discretization in space.
If $A$ is symmetric, 
which is natural for discretizations of $-\Delta$, 
then it can be diagonalized with $A=PDP^T$,
and the diagonalized system reads,
\begin{equation}\label{eq:sysODEreduced}
  \left\{
    \begin{aligned}
      \begin{pmatrix}
        \dot z_{i}\\
        \dot \mu_{i}
      \end{pmatrix}
      +
      \begin{pmatrix}
        d_i & -\nu^{-1} \\
        -1 & -d_i
      \end{pmatrix}
      \begin{pmatrix}
        z_{i}\\
        \mu_{i}
      \end{pmatrix}
      &=
      \begin{pmatrix}
        0\\
        -\hat z_{i}
      \end{pmatrix} \text{ in } (0,T),\\
      z_{i}(0) &= z_{i,0},\\
      \mu_{i}(T) + \gamma z_{i}(T) &= \gamma \hat z_{i}(T),
    \end{aligned}
  \right.
\end{equation}
where $d_i$ is the $i$-th eigenvalue of the matrix $A$, 
and $z_{i}$, 
$\mu_{i}$ as well as $\hat z_{i}$ are the $i$-th components of the vectors $\bm{z}$, 
$\bm{\mu}$ and $\hat{\bm{z}}$. 
Eliminating $\mu_{i}$ in~\eqref{eq:sysODEreduced}, 
we obtain the second-order ODE
\begin{equation}\label{eq:z}
  \left\{\begin{aligned}
    \ddot z_{i}- (d_i^2+\nu^{-1}) z_{i} &= -\nu^{-1}\hat z_{i} \text{ in } (0,T),\\
    z_{i}(0) &= z_{i,0},\\
    \dot z_{i}(T) + (\nu^{-1}\gamma + d_i) z_{i}(T)&= \nu^{-1}\gamma \hat z_{i}(T).
  \end{aligned}\right.
\end{equation}
We refer to~\cite[Section 2]{Gander2023} for more details about the transition from the PDE-constrained problem~\eqref{eq:J}-\eqref{eq:heat} to the diagonalized reduced problem~\eqref{eq:sysODEreduced}.

The rest of the paper is structured as follows.
We introduce in Section~\ref{sec:analysis} our new time decomposed Neumann-Neumann algorithms and study their convergence behavior in Section~\ref{sec:cv}. 
Numerical experiments are shown in
Section~\ref{sec:test} to support our analysis, 
and we draw conclusions in Section~\ref{sec:conclusion}.

\section{Neumann-Neumann algorithms}\label{sec:analysis}
In this section, 
we apply the Neumann-Neumann technique (NN) in time to obtain our new time domain decomposition methods to solve the system~\eqref{eq:sysODEreduced}, 
and investigate their convergence behavior. 
To focus on the error equation, 
we set both the initial condition $\y_0=0$ (i.e., $\bm{z}_0=0$) and the target function $\hat{\y}=0$ (i.e., $\hat{\bm{z}}=0$). 
We decompose the time domain $\Omega:=(0,T)$ into two non-overlapping subdomains $\Omega_1:=(0,\alpha)$ and $\Omega_2:=(\alpha,T)$, 
where $\alpha$ is the interface.
And we denote by $z_{j,i}$ and $\mu_{j,i}$ the restriction to $\Omega_j$, 
$j=1,2$ of the states $z_{i}$ and $\mu_{i}$. 
Although we will focus on the two-subdomain case in our current study,
the results can be extended to $N$ non-overlapping subdomains $\Omega_j:=(\alpha_j,\alpha_{j+1})$, 
$j=1,\ldots,N$ with $\alpha_1=0$ and $\alpha_{N+1}=T$.

Unlike the name of the NN algorithm suggests,
it starts first with a Dirichlet step,
which will be corrected by a Neumann step and then updates the transmission condition.
As the system~\eqref{eq:sysODEreduced} is a forward-backward system, 
it appears natural at first glance to keep this property for the decomposed case as illustrated in Figure~\ref{fig:forback}: 
we expect to have a final condition for the dual state $\mu_{1,i}$ in $\Omega_1$, 
since we already have an initial condition for $z_{1,i}$; 
similarly, 
we expect to have an initial condition for the primal state $z_{2,i}$ in $\Omega_2$, 
where we already have a final condition for $\mu_{2,i}$.
\begin{figure}
\label{fig:forback}
  \centering
  \begin{tikzpicture}
    \node (a) at (0,0) {0};
    \node (b) at (8,0) {$T$};
    \node (c) at (3,-0.2) {$\alpha$};
    \draw[thick,->] (a)--(b);
    \draw[dashed] (3.2,-0.6)--(3.2,1.1);
    \draw[<->] (0.2,0.8)--(3.2,0.8);
    \draw[<->] (3.2,0.6)--(7.7,0.6);
    \node (d) at (1.7,1) {$\Omega_1$};
    \node (e) at (5.5,0.8) {$\Omega_2$};
    \draw[->] (0.2,0.2)--(1,0.2);
    \node at (0.4,0.4) {$z_{i}$};
    \draw[dashed,->] (3.2,0.2)--(4,0.2);
    \draw[->] (7.6,-0.5)--(6.8,-0.5);
    \node at (7.4,-0.3) {$\mu_{i}$};
    \draw[dashed,->] (3.2,-0.5)--(2.4,-0.5);
  \end{tikzpicture}
  \caption{Illustration of the forward-backward system.}
\end{figure}
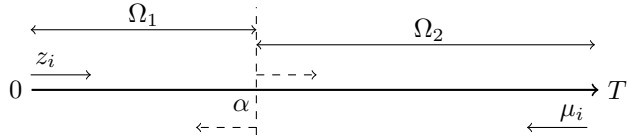
Therefore, 
for iteration index $k=1,2,\ldots$, 
a natural NN algorithm first solves the Dirichlet step
\begin{equation}\label{eq:NN1a}
\begin{aligned}
  &\left\{
    \begin{aligned}
      \begin{pmatrix}
        \dot z_{1,i}^k\\
        \dot \mu_{1,i}^k
      \end{pmatrix}
      +
      \begin{pmatrix}
        d_i & -\nu^{-1}\\
        -1 & -d_i
      \end{pmatrix}
      \begin{pmatrix}
        z_{1,i}^k\\
        \mu_{1,i}^k
      \end{pmatrix}
      &=
      \begin{pmatrix}
        0\\
        0
      \end{pmatrix} \text{ in } \Omega_1,\\
      z_{1,i}^k(0) &= 0,\\
      \mu_{1,i}^k(\alpha) & = f_{\alpha,i}^{k-1},
    \end{aligned}
  \right.\\
  &\left\{
    \begin{aligned}
      \begin{pmatrix}
        \dot z_{2,i}^k\\
        \dot \mu_{2,i}^k
      \end{pmatrix}
      +
      \begin{pmatrix}
        d_i & -\nu^{-1}\\
        -1 & -d_i
      \end{pmatrix}
      \begin{pmatrix}
        z_{2,i}^k\\
        \mu_{2,i}^k
      \end{pmatrix}
      &=
      \begin{pmatrix}
        0\\
        0
      \end{pmatrix} \text{ in } \Omega_2,\\
      z_{2,i}^k(\alpha) &= g_{\alpha,i}^{k-1},\\
      \mu_{2,i}^k(T) + \gamma z_{2,i}^k(T) &= 0,
    \end{aligned}
  \right.
   \end{aligned}
\end{equation}
then corrects the result by solving the Neumann step
\begin{equation}\label{eq:NN1acor}
\begin{aligned}
  &\left\{
    \begin{aligned}
      \begin{pmatrix}
        \dot \psi_{1,i}^k\\
        \dot \phi_{1,i}^k
      \end{pmatrix}
      +
      \begin{pmatrix}
        d_i & -\nu^{-1} \\
        -1 & -d_i
      \end{pmatrix}
      \begin{pmatrix}
        \psi_{1,i}^k\\
        \phi_{1,i}^k
      \end{pmatrix}
      &=
      \begin{pmatrix}
        0\\
        0
      \end{pmatrix} \text{ in } \Omega_1,\\
      \psi_{1,i}^k(0) &= 0,\\
      \dot\phi_{1,i}^k(\alpha) & = \dot \mu_{1,i}^k(\alpha) - \dot \mu_{2,i}^k(\alpha),
    \end{aligned}
  \right.\\
  &\left\{
    \begin{aligned}
      \begin{pmatrix}
        \dot \psi_{2,i}^k\\
        \dot \phi_{2,i}^k
      \end{pmatrix}
      +
      \begin{pmatrix}
        d_i & -\nu^{-1}\\
        -1 & -d_i
      \end{pmatrix}
      \begin{pmatrix}
        \psi_{2,i}^k\\
        \phi_{2,i}^k
      \end{pmatrix}
      &=
      \begin{pmatrix}
        0\\
        0
      \end{pmatrix} \text{ in } \Omega_2,\\
      \dot\psi_{2,i}^k(\alpha) &= \dot z_{2,i}^k(\alpha) - \dot z_{1,i}^k(\alpha),\\
      \phi_{2,i}^k(T) + \gamma \psi_{2,i}^k(T) &= 0,
    \end{aligned}
  \right.
\end{aligned}
\end{equation}
where $\psi_{i}$ is the primal correction state for $z_{i}$ and $\phi_{i}$ the dual correction state for $\mu_{i}$.
Finally, 
we update the transmission condition by 
\begin{equation}\label{eq:NN1atran}
	f_{\alpha,i}^k:=f_{\alpha,i}^{k-1} - \theta_1 \big(\phi_{1,i}^k(\alpha) + \phi_{2,i}^k(\alpha)\big), \quad g_{\alpha,i}^k:=g_{\alpha,i}^{k-1} - \theta_2 \big(\psi_{1,i}^k(\alpha) + \psi_{2,i}^k(\alpha)\big),
\end{equation}
with two relaxation parameters $\theta_1,\theta_2>0$.

As shown in the algorithm~\eqref{eq:NN1a}-\eqref{eq:NN1acor}, 
both Dirichlet and Neumann steps have the forward-backward structure. 
However, 
this structure only appears as being the natural one at first glance.
Indeed,
isolating the variable in each equation in the systems~\eqref{eq:NN1a} and~\eqref{eq:NN1acor}, 
we find the identities 
\begin{equation}\label{eq:4id} 
	\mu_{i} = \nu(\dot z_{i} + d_i z_{i}), \quad z_{i} = \dot \mu_{i}  - d_i \mu_{i}, \quad \phi_{i} = \nu(\dot \psi_{i} + d_i \psi_{i}), \quad \psi_{i} = \dot \phi_{i}  - d_i \phi_{i}.
\end{equation}
To shorten the notation,
we define
\begin{equation}\label{eq:sob}
  \sigma_i:=\sqrt{d_i^2+\nu^{-1}}, \quad \omega_i:=d_i+\gamma\nu^{-1}, \quad \beta_i:=1-\gamma d_i.
\end{equation}
Using~\eqref{eq:4id} and~\eqref{eq:sob}, 
we can rewrite the Dirichlet step~\eqref{eq:NN1a} in terms of the primal state $z_{i}$,
\begin{equation}\label{eq:errNN1}
  \left\{
    \begin{aligned}
      \ddot z_{1,i}^k - \sigma_i^2 z_{1,i}^k &= 0 \text{ in } \Omega_1,\\
      z_{1,i}^k(0) &= 0, \\
      \dot z_{1,i}^k(\alpha) + d_i z_{1,i}^k(\alpha) &= f_{\alpha,i}^{k-1},
    \end{aligned}
  \right.
  \quad
  \left\{
    \begin{aligned}
      \ddot z_{2,i}^k - \sigma_i^2 z_{2,i}^k &= 0 \text{ in } \Omega_2,\\
      z_{2,i}^k(\alpha) &= g_{\alpha,i}^{k-1},\\
      \dot z_{2,i}^k(T) + \omega_iz_{2,i}^k(T)&= 0.
    \end{aligned}
  \right.
\end{equation}
Similarly, 
the Neumann step~\eqref{eq:NN1acor} can be rewritten in terms of the primal correction state $\psi_{i}$,
\begin{equation}\label{eq:errNN1acor}
  \begin{aligned}
    &\left\{
      \begin{aligned}
        \ddot \psi_{1,i}^k - \sigma_i^2\psi_{1,i}^k &= 0 \text{ in } \Omega_1,\\
        \psi_{1,i}^k(0) &= 0, \\
        \dot \psi_{1,i}^k(\alpha) + \frac{\sigma_i^2}{d_i} \psi_{1,i}^k(\alpha) &= \big(\dot z_{1,i}^k(\alpha) + \frac{\sigma_i^2}{d_i} z_{1,i}^k(\alpha)\big) - \big(\dot z_{2,i}^k(\alpha) + \frac{\sigma_i^2}{d_i} z_{2,i}^k(\alpha)\big),
      \end{aligned}
    \right.
    \\
    &\left\{
      \begin{aligned}
        \ddot \psi_{2,i}^k - \sigma_i^2 \psi_{2,i}^k &= 0 \text{ in } \Omega_2,\\
        \dot \psi_{2,i}^k(\alpha) &= \dot z_{2,i}^k(\alpha) - \dot z_{1,i}^k(\alpha),\\
        \dot \psi_{2,i}^k(T) + \omega_i \psi_{2,i}^k(T) &= 0,
      \end{aligned}
    \right.
  \end{aligned}
\end{equation}
and the transmission condition~\eqref{eq:NN1atran} becomes
\begin{equation}\label{eq:errNN1atran}
  \begin{aligned}
    f_{\alpha,i}^k&=f_{\alpha,i}^{k-1} - \theta_1 \big(\dot \psi_{1,i}^k(\alpha) + d_i \psi_{1,i}^k(\alpha)+\dot \psi_{2,i}^k(\alpha) + d_i \psi_{2,i}^k(\alpha)\big), \\
    g_{\alpha,i}^k&=g_{\alpha,i}^{k-1} - \theta_2 \big(\psi_{1,i}^k(\alpha) + \psi_{2,i}^k(\alpha)\big).
  \end{aligned}
\end{equation}
Instead of using~\eqref{eq:NN1a}-\eqref{eq:NN1atran} for our analysis,
we will use the equivalent formulation in system~\eqref{eq:errNN1}-\eqref{eq:errNN1atran},  
in which the forward-backward structure has disappeared. 
Furthermore, 
the Dirichlet step in~\eqref{eq:NN1a} transforms in the primal state $z_{i}$ to a Robin-Dirichlet (RD) step~\eqref{eq:errNN1}, 
and the Neumann step in~\eqref{eq:NN1acor} transforms in the primal correction state $\psi_{i}$ to a Robin-Neumann (RN) step~\eqref{eq:errNN1acor}. 
In other words, 
we analyze actually a RD step with a RN correction, 
although it is originally a NN algorithm. 
We could also have interpreted the NN algorithm~\eqref{eq:NN1a}-\eqref{eq:NN1atran} using the dual state $\mu_{i}$ and the dual correction state $\phi_{i}$, 
the algorithm would then read differently but the convergence analysis is still the same (see~\cite{Gander2023}).
For the sake of consistency, 
we keep the interpretation with $z_{i}$ and $\psi_{i}$ for all convergence analyses.

The previous transformation reveals that the natural NN algorithm applied to the optimality system~\eqref{eq:sysODEreduced} is certainly not the only option.
Since there are three components in a NN algorithm:
a Dirichlet step, 
a Neumann step and an update step,
this expands our options when dealing with parabolic optimal control problems, 
and provides us with more choices within the NN algorithm.
More precisely,
instead of applying the Dirichlet step to the pair $(z_{i},\mu_{i})$, 
one can also apply it only to the primal state $z_{i}$ or the dual state $\mu_{i}$.
Likewise, 
the Neumann step can also be applied only to the primal correction state $\psi_{i}$ or the dual correction state $\phi_{i}$. 
We list in Table~\ref{tab:combination} all possible new time domain decomposition NN algorithms we can obtain, 
together with their equivalent interpretations in terms of the states $z_{i}$ and $\psi_{i}$. 
\begin{table}\label{tab:combination}
  \begin{center}
  \caption{Variants of the Neumann-Neumann algorithm.}
    \begin{tabular}{ c|c|c|c|c }
     category&step &$\Omega_1$ &$\Omega_2$& algorithm type  \\
     \hline\hline 
     \multirow{8}{*}{category I: $(z_{i},\mu_{i})$}&Dirichlet &$\mu_{i}$ &$z_{i}$ &(DD) \\  
     &step &$\dot z_{i}+d_i z_{i}$ &$z_{i}$ &(RD)\\
     \cline{2-5} 
      & &$\dot \phi_{i}$ &$\dot \psi_{i}$ &(NN)\\  
      & &$\ddot \psi_{i}+d_i \dot \psi_{i}$ &$\dot \psi_{i}$ &(RN)\\ 
     \cline{3-5}
     &Neumann &$\dot\psi_{i}$ &$\dot\psi_{i}$ &(NN) \\  
     &step &$\dot \psi_{i}$ &$\dot \psi_{i}$ &(NN)\\  
     \cline{3-5} 
     & &$\dot \phi_{i}$ &$\dot \phi_{i}$ &(NN)\\
     & &$\ddot \psi_{i}+d_i \dot \psi_{i}$ &$\ddot \psi_{i}+d_i \dot \psi_{i}$ &(RR)\\ 
     \hline\hline
     \multirow{8}{*}{category II: $z_{i}$}&Dirichlet &$z_{i}$ &$z_{i}$ &(DD) \\  
     &step &$z_{i}$ & $z_{i}$ &(DD)\\
     \cline{2-5} 
     & &$\dot\psi_{i}$ &$\dot\psi_{i}$ &(NN) \\    
     & &$\dot \psi_{i}$ &$\dot \psi_{i}$ &(NN)\\
     \cline{3-5}
     &Neumann &$\dot \phi_{i}$ &$\dot \psi_{i}$ &(NN)\\ 
     &step &$\ddot \psi_{i}+d_i \dot \psi_{i}$ &$\dot \psi_{i}$ &(RN)\\   
     \cline{3-5} 
     & &$\dot \phi_{i}$ &$\dot \phi_{i}$ &(NN)\\
     & &$\ddot \psi_{i}+d_i \dot \psi_{i}$ &$\ddot \psi_{i}+d_i \dot \psi_{i}$ &(RR)\\ 
     \hline\hline
     \multirow{8}{*}{category III: $\mu_{i}$}&Dirichlet &$\mu_{i}$ &$\mu_{i}$ &(DD) \\ 
     &step &$\dot z_{i}+d_i z_{i}$ &$\dot z_{i}+d_i z_{i}$ &(RR)\\
     \cline{2-5} 
     & &$\dot \phi_{i}$ &$\dot \phi_{i}$ &(NN)\\
     & &$\ddot \psi_{i}+d_i \dot \psi_{i}$ &$\ddot \psi_{i}+d_i \dot \psi_{i}$ &(RR)\\
     \cline{3-5}
     &Neumann &$\dot \phi_{i}$ &$\dot \psi_{i}$ &(NN)\\ 
     &step &$\ddot \psi_{i}+d_i \dot \psi_{i}$ &$\dot \psi_{i}$ &(RN)\\   
     \cline{3-5} 
     & &$\dot\psi_{i}$ &$\dot\psi_{i}$ &(NN) \\    
     & &$\dot \psi_{i}$ &$\dot \psi_{i}$ &(NN)\\
    \end{tabular}
  \end{center}
\end{table}
According to the Dirichlet step, 
they can be classified into three main categories. 
Each category is composed of two blocks, 
the first block represents the Dirichlet step and the second block the three possible Neumann steps. 
And each step contains two rows, 
the first row is the algorithm applied to~\eqref{eq:sysODEreduced}, 
and the second row represents the algorithm applied to~\eqref{eq:z}. 
Note that the update step should also be adapted when modifying the Dirichlet step or the Neumann step.
We will further discuss this in the next section, 
where we investigate the convergence of each algorithm.

\begin{remark}\label{rem:transfer}
	Although most of the algorithms in Table~\ref{tab:combination} do not look like having the forward-backward structure, 
	it can always be recovered by using the identities in~\eqref{eq:4id}. 
	Furthermore, 
	the transmission condition $\ddot \psi_{i} + d_i \dot \psi_{i}$ is actually a Robin type condition, 
	considering the first equation in~\eqref{eq:errNN1acor}. 
\end{remark}

\begin{remark}
	If the order in~\eqref{eq:NN1a}-\eqref{eq:NN1acor} is reversed,
	and one starts with the Neumann step,
	followed by the Dirichlet correction,
	the algorithm is then known under the name FETI (Finite Element Tearing and Interconnecting), 
	invented by Farhat and Roux~\cite{Farhat1991}.
	Since the two algorithms are very much related,
	we can also find similar variants as in Table~\ref{tab:combination} in the context of FETI algorithm.
\end{remark}

\section{Convergence analysis}\label{sec:cv}
In this section,
we will study the convergence of each algorithm listed in Table~\ref{tab:combination}.
Note that the two systems~\eqref{eq:errNN1} and~\eqref{eq:errNN1acor} are very similar, 
the only difference is in the transmission condition at $\alpha$. 
We can hence solve these two systems once and for all using the initial and the final condition, 
and find
\begin{equation}\label{eq:errNNsol}
	\begin{aligned}
		z_{1,i}^k(t) &= A_i^k \sinh(\sigma_i t), \quad z_{2,i}^k(t) =  B_i^k \Big(\sigma_i\cosh\big(\sigma_i (T-t)\big) + \omega_i\sinh\big(\sigma_i (T-t)\big) \Big),\\
		\psi_{1,i}^k(t) &= C_i^k \sinh(\sigma_i t), \quad \psi_{2,i}^k(t) = D_i^k\Big(\sigma_i\cosh\big(\sigma_i (T-t)\big) + \omega_i\sinh\big(\sigma_i (T-t)\big) \Big).
	\end{aligned}
\end{equation}
In general, 
the solutions~\eqref{eq:errNNsol} remain for all algorithms listed in Table~\ref{tab:combination}, 
and the coefficients $A_i^k, B_i^k, C_i^k$ and $D_i^k$ will be determined by the transmission conditions. 
To stay in a compact form,
we will only present the modified step for each NN variant instead of giving a complete three-step algorithm.

\subsection{Category I}
This category consists in applying the Dirichlet step to the pair $(z_{i},\mu_{i})$. 
As illustrated in Table~\ref{tab:combination}, 
there are three variants according to the Neumann correction step.

\subsubsection{Algorithm NN$_{1\text{a}}$}
This is~\eqref{eq:NN1a}-\eqref{eq:NN1atran}, 
at first glance the most natural NN algorithm, 
which keeps the forward-backward structure both for the Dirichlet and Neumann steps.
To analyze its convergence behavior, 
we interpret it as~\eqref{eq:errNN1}-\eqref{eq:errNN1atran} and solve for the exact iterates. 
Using~\eqref{eq:errNNsol}, 
we determine the coefficients $A_i^k$, $B_i^k$ through the transmission conditions in~\eqref{eq:errNN1}, 
and find 
\begin{equation}\label{eq:D1AB}
	A_i^k = \frac{f_{\alpha,i}^{k-1}}{\sigma_i\cosh(a_i) + d_i\sinh(a_i)}, \quad B_i^k = \frac{g_{\alpha,i}^{k-1}}{\sigma_i\cosh(b_i) + \omega_i\sinh(b_i) },
\end{equation}
where we let $a_i:=\sigma_i\alpha$ and $b_i:=\sigma_i (T-\alpha)$ to simplify the notations, 
and $a_i+b_i = \sigma_i T$. 
Using once again~\eqref{eq:errNNsol}, 
we determine the coefficients $C_i^k$, $D_i^k$ through the transmission conditions in~\eqref{eq:errNN1acor}
\begin{equation}\label{eq:N1CD}
	C_i^k = A_i^k - B_i^k\nu^{-1}\frac{\sigma_i\gamma\sinh(b_i) + \beta_i\cosh(b_i)}{\sigma_i\sinh(a_i)+d_i\cosh(a_i)}, \, D_i^k=A_i^k\frac{\cosh(a_i)}{\sigma_i\sinh(b_i) + \omega_i\cosh(b_i)} + B_i^k.
\end{equation}
We then update the transmission condition~\eqref{eq:errNN1atran}, and find
\begin{equation}\label{eq:NN1aiter}
  \begin{pmatrix}
    f_{\alpha,i}^k\\
    g_{\alpha,i}^k
  \end{pmatrix} = \begin{pmatrix}
    1-\theta_1d_iE_i & \theta_1\nu^{-1} F_i\\
    -\theta_2E_i & 1-\theta_2d_iF_i
  \end{pmatrix}\begin{pmatrix}
    f_{\alpha,i}^{k-1}\\
    g_{\alpha,i}^{k-1}
  \end{pmatrix},
\end{equation}
with 
\[\begin{aligned}
E_i=\frac{\sigma_i\cosh(\sigma_iT)+\omega_i\sinh(\sigma_iT)}{\sigma_i\sinh(b_i)+\omega_i\cosh(b_i)}\frac{1}{\sigma_i\cosh(a_i)+d_i\sinh(a_i)},\\
F_i=\frac{\sigma_i\cosh(\sigma_i T)+\omega_i\sinh(\sigma_i T)}{\sigma_i\cosh(b_i) + \omega_i\sinh(b_i)}\frac{1}{\sigma_i\sinh(a_i)+d_i\cosh(a_i)}.
\end{aligned}\]
The characteristic polynomial associated with the iteration matrix in~\eqref{eq:NN1aiter} is 
\[X^2 + (\theta_1d_iE_i +\theta_2d_iF_i-2)X + 1-\theta_1d_iE_i-\theta_2d_iF_i +\theta_1\theta_2\sigma_i^2E_iF_i.\]
We then have the following result.

\begin{theorem}
  Algorithm NN$_{1\text{a}}$~\eqref{eq:NN1a}-\eqref{eq:NN1atran} converges if and only if 
  \begin{equation}\label{eq:rhoNN1a}
      \rho_{\text{NN}_{1\text{a}}}:=\max_{d_i\in\lambda(A)}\Big\{\Big|1-\frac{d_i(\theta_1E_i +\theta_2F_i)\pm\sqrt{d_i^2(\theta_1E_i +\theta_2F_i)^2  -4\theta_1\theta_2\sigma_i^2E_iF_i}}{2}\Big|\Big\}<1,
  \end{equation}
  where $\lambda(A)$ is the spectrum of the matrix $A$.
\end{theorem}

To get more insight in the convergence factor~\eqref{eq:rhoNN1a}, 
we consider a few special cases.
Supposing no final target (i.e., $\gamma=0$) and a symmetric decomposition $\alpha=\frac T2$ (i.e., $a_i=b_i$),
we have 
\[E_i=F_i=\frac{2d_i\tanh(a_i)+\sigma_i(1+\tanh^2(a_i))}{(\sigma_i^2+d_i^2)\tanh(a_i)+d_i\sigma_i (1+\tanh^2(a_i))}<\frac 1{d_i}.\] 
Letting $\theta_1=\theta_2=\theta$, 
the convergence factor~\eqref{eq:rhoNN1a} then becomes
\[|1-\theta d_iE_i\pm\theta E_i\sqrt{d_i^2-\sigma_i^2}|,\]
where the discriminant is negative due to $d_i^2-\sigma_i^2 = -\nu^{-1}$.
Thus,
the convergence factor $\rho_{\text{NN}_{1\text{a}}}$ in this case is
\[\sqrt{1-2\theta d_i E_i+\theta^2\sigma_i^2E_i^2}>\sqrt{1-2\theta+\theta^2\sigma_i^2E_i^2}\geq\sqrt{1-2\theta}.\]

\begin{remark}\label{rem:0}
	For the Laplace operator with homogeneous Dirichlet boundary conditions in our model problem~\eqref{eq:heat}, 
	there is no zero eigenvalue for its discretization matrix $A$.
	For a zero eigenvalue, 
	$d_i=0$, 
	we have from~\eqref{eq:sob} that 
	\begin{equation}\label{eq:sob0}
		\sigma_i|_{d_i=0}=\sqrt{\nu^{-1}}, \quad \omega_i|_{d_i=0}=\gamma\nu^{-1}, \quad \beta_i|_{d_i=0}=1.
	\end{equation} 
	Substituting~\eqref{eq:sob0} into the convergence factor~\eqref{eq:rhoNN1a},
	we find $\rho_{\text{NN}_{1\text{a}}}|_{d_i=0}=\{ |1\pm\sqrt{-\theta_1\theta_2(E_iF_i)|_{d_i=0}} |\}$ with 
\[\begin{aligned}
(E_iF_i)|_{d_i=0}=2 + \coth(\sqrt{\nu^{-1}}\alpha) \frac{\coth(\sqrt{\nu^{-1}}(T-\alpha))+\gamma\sqrt{\nu^{-1}}}{1+\gamma\sqrt{\nu^{-1}}\coth(\sqrt{\nu^{-1}}(T-\alpha)) } \\
+\tanh(\sqrt{\nu^{-1}}\alpha) \frac{\tanh(\sqrt{\nu^{-1}}(T-\alpha))+\gamma\sqrt{\nu^{-1}}}{1+ \gamma\sqrt{\nu^{-1}}\tanh(\sqrt{\nu^{-1}}(T-\alpha))}.
\end{aligned}\]
	Since $(E_iF_i)|_{d_i=0}$, $\theta_1$, $\theta_2$ are all positive,
	the discriminant is once again negative,
	and we have $\rho_{\text{NN}_{1\text{a}}}|_{d_i=0}=\sqrt{1+ \theta_1\theta_2(E_iF_i)|_{d_i=0}}$,
	which is always greater than one.
	In other words,
	the convergence behavior of algorithm NN$_{1\text{a}}$ for small eigenvalues is not good,
	and cannot be fixed with relaxation.
\end{remark}

\begin{remark}\label{rem:infty}
  For large eigenvalues $d_i$, 
  we have from~\eqref{eq:sob} that 
  \begin{equation}\label{eq:sobinf}
  	\sigma_i\sim_{\infty} d_i, \quad \omega_i\sim_{\infty} d_i, \quad \beta_i\sim_{\infty} -d_i,
   \end{equation} 
  and thus obtain $E_i \sim_{\infty} \frac1{d_i}$ and $F_i \sim_{\infty} \frac1{d_i}$. 
  Substituting these into~\eqref{eq:rhoNN1a}, 
  we find $\lim_{d_i\rightarrow\infty} \rho_{\text{NN}_{1\text{a}}} = \{|1-\theta_1|,|1-\theta_2|\}$.
  In other words, 
  high frequency convergence is robust with relaxation,
  and one can get a good smoother using $\theta_1 = \theta_2 = 1$.
\end{remark}

The above analysis reveals the fact that this most natural NN algorithm is a good smoother but not a good solver.

\subsubsection{Algorithm NN$_{1\text{b}}$}\label{sec:NN1b}
We apply now the Neumann step only to the primal correction state $\psi_{i}$. 
For $k=1,2,\ldots$, 
we consider the algorithm that first solves the Dirichlet step~\eqref{eq:NN1a}, 
and then corrects it by solving the Neumann step
\begin{equation}\label{eq:NN1bcor}
\begin{aligned}
  &\left\{
    \begin{aligned}
      \begin{pmatrix}
        \dot \psi_{1,i}^k\\
        \dot \phi_{1,i}^k
      \end{pmatrix}
      +
      \begin{pmatrix}
        d_i & -\nu^{-1} \\
        -1 & -d_i
      \end{pmatrix}
      \begin{pmatrix}
        \psi_{1,i}^k\\
        \phi_{1,i}^k
      \end{pmatrix}
      &=
      \begin{pmatrix}
        0\\
        0
      \end{pmatrix} \text{ in } \Omega_1,\\
      \psi_{1,i}^k(0) &= 0,\\
      \dot\psi_{1,i}^k(\alpha) & = \dot z_{1,i}^k(\alpha) - \dot z_{2,i}^k(\alpha),
    \end{aligned}
  \right.\\
  &\left\{
    \begin{aligned}
      \begin{pmatrix}
        \dot \psi_{2,i}^k\\
        \dot \phi_{2,i}^k
      \end{pmatrix}
      +
      \begin{pmatrix}
        d_i & -\nu^{-1}\\
        -1 & -d_i
      \end{pmatrix}
      \begin{pmatrix}
        \psi_{2,i}^k\\
        \phi_{2,i}^k
      \end{pmatrix}
      &=
      \begin{pmatrix}
        0\\
        0
      \end{pmatrix} \text{ in } \Omega_2,\\
      \dot\psi_{2,i}^k(\alpha) &= \dot z_{2,i}^k(\alpha) - \dot z_{1,i}^k(\alpha),\\
      \phi_{2,i}^k(T) + \gamma \psi_{2,i}^k(T) &= 0.
    \end{aligned}
  \right.
\end{aligned}
\end{equation}
As for the update step,
let us first consider keeping the same update as~\eqref{eq:NN1atran}.

Unlike the Dirichlet step~\eqref{eq:NN1a}, 
the Neumann step~\eqref{eq:NN1bcor} does not have the forward-backward structure in the current form, 
but this can be recovered using the identities in~\eqref{eq:4id}. 
More precisely, 
we can rewrite the transmission condition 
$\dot\psi_{1,i}^k(\alpha) = \dot z_{1,i}^k(\alpha) - \dot z_{2,i}^k(\alpha)$ as 
\[\dot\phi_{1,i}^k(\alpha)-\frac{\sigma_i^2}{d_i}\phi_{1,i}^k(\alpha)   = (\dot\mu_{1,i}^k(\alpha) - \frac{\sigma_i^2}{d_i} \mu_{1,i}^k(\alpha)) - (\dot\mu_{2,i}^k(\alpha) - \frac{\sigma_i^2}{d_i} \mu_{2,i}^k(\alpha)),\] 
which is a Robin type condition.
In other words, 
when the forward-backward structure is recovered with this interpretation, 
the Neumann step~\eqref{eq:NN1bcor} becomes a RN step.

Compared with algorithm NN$_{1\text{a}}$, 
only the Neumann step is modified, 
which can be transformed into  
\begin{equation}\label{eq:errNN1bcor}
  \begin{aligned}
    \left\{
      \begin{aligned}
        \ddot \psi_{1,i}^k - \sigma_i^2\psi_{1,i}^k &= 0 \text{ in } \Omega_1,\\
        \psi_{1,i}^k(0) &= 0, \\
        \dot \psi_{1,i}^k(\alpha)  &= \dot z_{1,i}^k(\alpha)  - \dot z_{2,i}^k(\alpha),
      \end{aligned}
    \right.
    \quad
    \left\{
      \begin{aligned}
        \ddot \psi_{2,i}^k - \sigma_i^2 \psi_{2,i}^k &= 0 \text{ in } \Omega_2,\\
        \dot \psi_{2,i}^k(\alpha) &= \dot z_{2,i}^k(\alpha) - \dot z_{1,i}^k(\alpha),\\
        \dot \psi_{2,i}^k(T) + \omega_i \psi_{2,i}^k(T) &= 0.
      \end{aligned}
    \right.
  \end{aligned}
\end{equation}
The convergence analysis is then given by solving explicitly~\eqref{eq:errNN1}, 
\eqref{eq:errNN1bcor} and~\eqref{eq:errNN1atran} for one step.
In this form, 
we are actually analyzing here a RD step with a NN correction step.
Using~\eqref{eq:errNNsol}, 
we can solve~\eqref{eq:errNN1bcor} and determine the coefficients 
\begin{equation}\label{eq:N2CD}
	C_i^k = A_i^k + B_i^k \frac{\sigma_i\sinh(b_i) + \omega_i\cosh(b_i)}{\cosh(a_i)}, \, D_i^k=A_i^k\frac{\cosh(a_i)}{\sigma_i\sinh(b_i) + \omega_i\cosh(b_i)} + B_i^k.
\end{equation} 
Combining with~\eqref{eq:D1AB}, 
we update the transmission condition~\eqref{eq:errNN1atran} and find
\begin{equation}\label{eq:NN1biter}
  \begin{pmatrix}
    f_{\alpha,i}^k\\
    g_{\alpha,i}^k
  \end{pmatrix} = \begin{pmatrix}
    1-\theta_1d_iE_i & -\theta_1d_i F_i\\
    -\theta_2E_i & 1-\theta_2F_i
  \end{pmatrix}\begin{pmatrix}
    f_{\alpha,i}^{k-1}\\
    g_{\alpha,i}^{k-1}
  \end{pmatrix},
\end{equation}
with 
\[\begin{aligned}
E_i&=\frac{\sigma_i\cosh(\sigma_iT)+\omega_i\sinh(\sigma_iT)}{\sigma_i\sinh(b_i)+ \omega_i\cosh(b_i)}\frac{1}{\sigma_i \cosh(a_i) + d_i \sinh(a_i)}, \\
F_i&=\frac{\sigma_i\cosh(\sigma_iT)+\omega_i\sinh(\sigma_iT)}{\sigma_i\cosh(b_i) + \omega_i\sinh(b_i)}\frac1{\cosh(a_i)}.
\end{aligned}\] 
In particular, 
the eigenvalues of the iteration matrix in~\eqref{eq:NN1biter} are 1 and $1-(\theta_1d_iE_i+\theta_2F_i)$, 
meaning that the algorithm~\eqref{eq:NN1a}, \eqref{eq:NN1bcor}, \eqref{eq:NN1atran} stagnates in its current form,
and cannot be fixed even with relaxation.

Note that we choose to keep the same Dirichlet and update steps in the algorithm~\eqref{eq:NN1a}, \eqref{eq:NN1bcor}, \eqref{eq:NN1atran}, 
although the Neumann step has been changed comparing to algorithm NN$_{1\text{a}}$.
We also observe from the Neumann correction step~\eqref{eq:NN1bcor} that $\dot \psi_{1,i}^k(\alpha)  + \dot \psi_{2,i}^k(\alpha)=0$, 
which implies that in this case,
the update step~\eqref{eq:NN1atran} in terms of the primal correction state~\eqref{eq:errNN1atran} is actually
\begin{equation}\label{eq:errNN1atran_in_1b}
	f_{\alpha,i}^k= f_{\alpha,i}^{k-1} - \theta_1d_i\big(\psi_{1,i}^k(\alpha) + \psi_{2,i}^k(\alpha)\big), \quad g_{\alpha,i}^k=g_{\alpha,i}^{k-1} - \theta_2 \big(\psi_{1,i}^k(\alpha) + \psi_{2,i}^k(\alpha)\big).
\end{equation}
In other words,
we update both $f_{\alpha,i}^k$ and $g_{\alpha,i}^k$ only by $\psi_{i}^k(\alpha)$.
This observation leads to the idea to consider a modified NN algorithm.
More precisely,
we first remove $d_i$ in~\eqref{eq:errNN1atran_in_1b} as
\begin{equation}\label{eq:tran_case1}
	f_{\alpha,i}^{k}=f_{\alpha,i}^{k-1}-\theta_1 (\psi_{1,i}^k(\alpha)  + \psi_{2,i}^k(\alpha)), \quad g_{\alpha,i}^{k}=g_{\alpha,i}^{k-1}-\theta_2 (\psi_{1,i}^k(\alpha)  + \psi_{2,i}^k(\alpha)).
\end{equation}
In the case when $f_{\alpha,i}^{0}=g_{\alpha,i}^{0}$ and $\theta_1=\theta_2=\theta$, 
we have $f_{\alpha,i}^{k}=g_{\alpha,i}^{k}$, 
$\forall k\in\mathbb{N}$.
In this way, 
we consider the modified NN algorithm which solves first the Dirichlet step
\begin{equation}\label{eq:NN1b}
\begin{aligned}
  &\left\{
    \begin{aligned}
      \begin{pmatrix}
        \dot z_{1,i}^k\\
        \dot \mu_{1,i}^k
      \end{pmatrix}
      +
      \begin{pmatrix}
        d_i & -\nu^{-1}\\
        -1 & -d_i
      \end{pmatrix}
      \begin{pmatrix}
        z_{1,i}^k\\
        \mu_{1,i}^k
      \end{pmatrix}
      &=
      \begin{pmatrix}
        0\\
        0
      \end{pmatrix} \text{ in } \Omega_1,\\
      z_{1,i}^k(0) &= 0,\\
      \mu_{1,i}^k(\alpha) & = f_{\alpha,i}^{k-1},
    \end{aligned}
  \right.\\
  &\left\{
    \begin{aligned}
      \begin{pmatrix}
        \dot z_{2,i}^k\\
        \dot \mu_{2,i}^k
      \end{pmatrix}
      +
      \begin{pmatrix}
        d_i & -\nu^{-1}\\
        -1 & -d_i
      \end{pmatrix}
      \begin{pmatrix}
        z_{2,i}^k\\
        \mu_{2,i}^k
      \end{pmatrix}
      &=
      \begin{pmatrix}
        0\\
        0
      \end{pmatrix} \text{ in } \Omega_2,\\
      z_{2,i}^k(\alpha) &= f_{\alpha,i}^{k-1},\\
      \mu_{2,i}^k(T) + \gamma z_{2,i}^k(T) &= 0,
    \end{aligned}
  \right.
   \end{aligned}
\end{equation}
then corrects the result by solving the Neumann step~\eqref{eq:NN1bcor} and updates the transmission condition by 
\begin{equation}\label{eq:NN1btran}
	f_{\alpha,i}^{k}=f_{\alpha,i}^{k-1}-\theta (\psi_{1,i}^k(\alpha)  + \psi_{2,i}^k(\alpha)), \quad \theta>0.
\end{equation}
For this modified NN algorithm, 
we find the following result.

\begin{theorem}\label{thm:NN1b}
	Algorithm NN$_{1\text{b}}$~\eqref{eq:NN1b}, \eqref{eq:NN1bcor}, \eqref{eq:NN1btran} converges if and only if
	\begin{equation}\label{eq:rhoNN1b}
		\rho_{\text{NN}_{1\text{b}}}:=\max_{d_i\in\lambda(A)}\big|1-\theta(E_i+F_i)\big|<1.
	\end{equation}
\end{theorem}

Compared to the algorithm~\eqref{eq:NN1a}, \eqref{eq:NN1bcor}, \eqref{eq:NN1atran}, 
algorithm NN$_{1\text{b}}$ converges with a proper choice of $\theta$. 
More precisely,
for a zero eigenvalue, 
substituting~\eqref{eq:sob0} into~\eqref{eq:rhoNN1b}, 
we find
$\dot\psi_{1,i}^k(\alpha) = \dot z_{1,i}^k(\alpha) - \dot z_{2,i}^k(\alpha)$ as 
\[\dot\phi_{1,i}^k(\alpha)-\frac{\sigma_i^2}{d_i}\phi_{1,i}^k(\alpha)   = (\dot\mu_{1,i}^k(\alpha) - \frac{\sigma_i^2}{d_i} \mu_{1,i}^k(\alpha)) - (\dot\mu_{2,i}^k(\alpha) - \frac{\sigma_i^2}{d_i} \mu_{2,i}^k(\alpha)),\] 
meaning that small eigenvalue convergence is good with relaxation.
For large eigenvalues $d_i$, 
using~\eqref{eq:sobinf},
we have $E_i\sim_{\infty}\frac1{d_i}$ and $F_i\sim_{\infty}2$.
Thus, 
we obtain $\lim_{d_i\rightarrow \infty} \rho_{\text{NN}_{1\text{b}}}=|1-2\theta|$, 
which is independent of the interface $\alpha$. 
So high frequency convergence is robust with relaxation, 
and one can get a good smoother using $\theta=1/2$.
By equioscillating the convergence factor for small (i.e., $\rho_{\text{NN}_{1\text{b}}} |_{d_i=0}$) and large (i.e., $\rho_{\text{NN}_{1\text{b}}} |_{d_i\rightarrow\infty}$) eigenvalues, 
we obtain
\begin{equation}\label{eq:thetaNN1b}
	{\scriptstyle 
		\theta^*_{\text{NN}_{1\text{b}}} := \frac{2}{3+\sqrt{\nu}(\tanh(\sqrt{\nu^{-1}}\alpha)+\frac{1+\gamma\sqrt{\nu^{-1}}\tanh(\sqrt{\nu^{-1}}(T-\alpha))}{\gamma\sqrt{\nu^{-1}}+\tanh(\sqrt{\nu^{-1}}(T-\alpha))})+\tanh(\sqrt{\nu^{-1}}\alpha)\frac{\gamma\sqrt{\nu^{-1}} +\tanh(\sqrt{\nu^{-1}}(T-\alpha))}{1 + \gamma\sqrt{\nu^{-1}}\tanh(\sqrt{\nu^{-1}}(T-\alpha))}} 
	},
\end{equation}
which is smaller than 2/3. 
However, 
it is not clear under what condition $\theta^*_{\text{NN}_{1\text{b}}}$ is the optimal relaxation parameter.
Indeed,
the monotonicity of $E_i$ and $F_i$ with respect to $d_i$ may change according to the parameter values $\alpha$, 
$\gamma$ and $\nu$.
Thus,
the variation of $E_i+F_i$ to $d_i$ is less clear even in the case with $\gamma=0$.
Generally, 
algorithm NN$_{1\text{b}}$ is a good smoother and can also be a good solver with a proper relaxation parameter $\theta$.

\begin{remark}\label{rem:tran_NN1b}
	Instead of considering the update step as in~\eqref{eq:tran_case1},
	we could have also modified~\eqref{eq:errNN1atran_in_1b} to
\[f_{\alpha,i}^{k}=f_{\alpha,i}^{k-1}-\theta_1 d_i(\psi_{1,i}^k(\alpha)  + \psi_{2,i}^k(\alpha)), g_{\alpha,i}^{k}=g_{\alpha,i}^{k-1}-\theta_2 d_i (\psi_{1,i}^k(\alpha)  + \psi_{2,i}^k(\alpha)).\]
	Using then the same arguments as above,
	we end up with 
	$g_{\alpha,i}^{k}\equiv f_{\alpha,i}^{k}=f_{\alpha,i}^{k-1}(1-\theta d_i(E_i+F_i))$.
	However, 
	the convergence of the algorithm can no longer be guaranteed with this update.
	More precisely, 
	for a zero eigenvalue $d_i=0$,
	the convergence factor is one,
	and cannot be improved with relaxation.
	As for large eigenvalues,
	using once again the equivalence relation of $E_i$ and $F_i$,
	we find the convergence factor goes to infinity when $d_i$ is large.
\end{remark}

In general,
the above analysis shows that the update step should also be adapted when modifying the Neumann step.

\subsubsection{Algorithm NN$_{1\text{c}}$}
Instead of applying the Neumann step to the primal correction state $\psi_{i}$, 
we can also apply it only to the dual correction state $\phi_{i}$. 
For $k=1,2,\ldots$, 
we consider the algorithm that first solves the Dirichlet step~\eqref{eq:NN1a}, 
then corrects it by solving the Neumann step
\begin{equation}\label{eq:NN1ccor}
\begin{aligned}
  &\left\{
    \begin{aligned}
      \begin{pmatrix}
        \dot \psi_{1,i}^k\\
        \dot \phi_{1,i}^k
      \end{pmatrix}
      +
      \begin{pmatrix}
        d_i & -\nu^{-1} \\
        -1 & -d_i
      \end{pmatrix}
      \begin{pmatrix}
        \psi_{1,i}^k\\
        \phi_{1,i}^k
      \end{pmatrix}
      &=
      \begin{pmatrix}
        0\\
        0
      \end{pmatrix} \text{ in } \Omega_1,\\
      \psi_{1,i}^k(0) &= 0,\\
      \dot\phi_{1,i}^k(\alpha) & = \dot \mu_{1,i}^k(\alpha) - \dot \mu_{2,i}^k(\alpha),
    \end{aligned}
  \right.\\
  &\left\{
    \begin{aligned}
      \begin{pmatrix}
        \dot \psi_{2,i}^k\\
        \dot \phi_{2,i}^k
      \end{pmatrix}
      +
      \begin{pmatrix}
        d_i & -\nu^{-1}\\
        -1 & -d_i
      \end{pmatrix}
      \begin{pmatrix}
        \psi_{2,i}^k\\
        \phi_{2,i}^k
      \end{pmatrix}
      &=
      \begin{pmatrix}
        0\\
        0
      \end{pmatrix} \text{ in } \Omega_2,\\
      \dot\phi_{2,i}^k(\alpha) &= \dot \mu_{2,i}^k(\alpha) - \dot \mu_{1,i}^k(\alpha),\\
      \phi_{2,i}^k(T) + \gamma \psi_{2,i}^k(T) &= 0.
    \end{aligned}
  \right.
\end{aligned}
\end{equation}
Once again,
let us first consider keeping the same update step~\eqref{eq:NN1atran}.

The Neumann step~\eqref{eq:NN1ccor} does not seem to have the forward-backward structure due to the transmission condition on the second domain $\Omega_2$. 
Using~\eqref{eq:4id}, 
we can rewrite it as 
\[\dot \psi_{2,i}^k(\alpha) + \frac{\sigma_i^2}{d_i} \psi_{1,i}^k(\alpha) = (\dot z_{2,i}^k(\alpha) + \frac{\sigma_i^2}{d_i} z_{2,i}^k(\alpha)) - (\dot z_{1,i}^k(\alpha) + \frac{\sigma_i^2}{d_i} z_{1,i}^k(\alpha)),\]
which then becomes a NR step with the usual forward-backward structure.

Once again, 
only the Neumann step is modified and can be transformed into  
\begin{equation}\label{eq:errNN1ccor}
  \begin{aligned}
    &\left\{
      \begin{aligned}
        \ddot \psi_{1,i}^k - \sigma_i^2\psi_{1,i}^k &= 0 \text{ in } \Omega_1,\\
        \psi_{1,i}^k(0) &= 0, \\
        \dot \psi_{1,i}^k(\alpha) + \frac{\sigma_i^2}{d_i} \psi_{1,i}^k(\alpha) &= \big(\dot z_{1,i}^k(\alpha) + \frac{\sigma_i^2}{d_i} z_{1,i}^k(\alpha)\big) - \big(\dot z_{2,i}^k(\alpha) + \frac{\sigma_i^2}{d_i} z_{2,i}^k(\alpha)\big),
      \end{aligned}
    \right.
    \\
    &\left\{
      \begin{aligned}
        \ddot \psi_{2,i}^k - \sigma_i^2 \psi_{2,i}^k &= 0 \text{ in } \Omega_2,\\
        \dot \psi_{2,i}^k(\alpha) + \frac{\sigma_i^2}{d_i} \psi_{1,i}^k(\alpha) &= \big(\dot z_{2,i}^k(\alpha) + \frac{\sigma_i^2}{d_i} z_{2,i}^k(\alpha)\big) - \big(\dot z_{1,i}^k(\alpha) + \frac{\sigma_i^2}{d_i} z_{1,i}^k(\alpha)\big),\\
        \dot \psi_{2,i}^k(T) + \omega_i \psi_{2,i}^k(T) &= 0.
      \end{aligned}
    \right.
  \end{aligned}
\end{equation}
The convergence analysis is thus given for a RD step~\eqref{eq:errNN1} with a RR correction step~\eqref{eq:errNN1ccor}.
We can solve~\eqref{eq:errNN1ccor} using~\eqref{eq:errNNsol} and determine the coefficients 
\begin{equation}\label{eq:N3CD}
	C_i^k= A_i^k - B_i^k \nu^{-1}\frac{\sigma_i\gamma\sinh(b_i) +\beta_i\cosh(b_i)}{\sigma_i \sinh(a_i)+d_i\cosh(a_i)}, D_i^k=B_i^k -\nu A_i^k\frac{\sigma_i \sinh(a_i)+d_i\cosh(a_i) }{\sigma_i\gamma\sinh(b_i) + \beta_i\cosh(b_i)}.
\end{equation} 
Combining with~\eqref{eq:D1AB},
we update the transmission condition~\eqref{eq:errNN1atran} and find
\begin{equation}\label{eq:NN1citer}
  \begin{pmatrix}
    f_{\alpha,i}^k\\
    g_{\alpha,i}^k
  \end{pmatrix} = \begin{pmatrix}
    1-\theta_1E_i & \theta_1\nu^{-1} F_i\\
    \theta_2\nu d_iE_i & 1-\theta_2d_iF_i
  \end{pmatrix}\begin{pmatrix}
    f_{\alpha,i}^{k-1}\\
    g_{\alpha,i}^{k-1}
  \end{pmatrix},
\end{equation}
with 
\[\begin{aligned}
E_i=\frac{\sigma_i\cosh(\sigma_iT)+\omega_i\sinh(\sigma_iT)}{\sigma_i\gamma\sinh(b_i)+ \beta_i\cosh(b_i)}\frac{1}{\sigma_i \cosh(a_i) + d_i \sinh(a_i)},\\ 
F_i=\frac{\sigma_i\cosh(\sigma_iT)+\omega_i\sinh(\sigma_iT)}{\sigma_i\cosh(b_i) + \omega_i\sinh(b_i)}\frac1{\sigma_i\sinh(a_i)+d_i\cosh(a_i)}.
\end{aligned}\] 
In particular, 
the eigenvalues of the iteration matrix in~\eqref{eq:NN1citer} are 1 and $1-(\theta_1E_i+\theta_2d_iF_i)$. 
Once again,
the algorithm~\eqref{eq:NN1a}, \eqref{eq:NN1ccor}, \eqref{eq:NN1atran} stagnates, and cannot be fixed with relaxation.
Similar as in Section~\ref{sec:NN1b}, 
we can adapt the transmission condition~\eqref{eq:NN1atran} and make this algorithm converge.
More precisely, 
we first consider the update 
\[f_{\alpha,i}^{k}=f_{\alpha,i}^{k-1}-\theta (\phi_{1,i}^k(\alpha)  + \phi_{2,i}^k(\alpha)), g_{\alpha,i}^{k}=g_{\alpha,i}^{k-1}-\theta (\phi_{1,i}^k(\alpha)  + \phi_{2,i}^k(\alpha)).\]
In the case when $f_{\alpha,i}^{0}=g_{\alpha,i}^{0}$ and $\theta_1=\theta_2=\theta$, 
we have $g_{\alpha,i}^{k}= f_{\alpha,i}^{k}$, $\forall k\in\mathbb{N}$ and
\begin{equation}\label{eq:NN1ctran}
	f_{\alpha,i}^{k}=f_{\alpha,i}^{k-1}-\theta (\phi_{1,i}^k(\alpha)  + \phi_{2,i}^k(\alpha)).
\end{equation}
This leads to the following result.

\begin{theorem}\label{thm:NN1c}
	Algorithm NN$_{1\text{c}}$~\eqref{eq:NN1b}, \eqref{eq:NN1ccor}, \eqref{eq:NN1ctran} converges if and only if
	\begin{equation}\label{eq:rhoNN1c}
	\rho_{\text{NN}_{1\text{c}}}:=\max_{d_i\in\lambda(A)}|1-\theta(E_i-\nu^{-1}F_i)|<1.
	\end{equation}
\end{theorem}

Compared to the algorithm~\eqref{eq:NN1a}, \eqref{eq:NN1ccor}, \eqref{eq:NN1atran}, 
algorithm NN$_{1\text{c}}$ may converge with a proper choice of $\theta$. 
More precisely,
for a zero eigenvalue, 
$d_i=0$, 
we find 
\[\begin{aligned}
\rho_{\text{NN}_{1\text{c}}}|_{d_i=0} =| 1-\theta (1+\tanh(\sqrt{\nu^{-1}}\alpha)\frac{ \gamma\sqrt{\nu^{-1}} + \tanh(\sqrt{\nu^{-1}}(T-\alpha)) }{\gamma\sqrt{\nu^{-1}}\tanh(\sqrt{\nu^{-1}}(T-\alpha)) + 1} \\
-  \sqrt{\nu^{-1}}(\coth(\sqrt{\nu^{-1}}\alpha)  + \frac{\gamma\sqrt{\nu^{-1}} + \tanh(\sqrt{\nu^{-1}}(T-\alpha))}{1+ \gamma\sqrt{\nu^{-1}}\tanh(\sqrt{\nu^{-1}}(T-\alpha))}))|.
\end{aligned}\]
Depending on the values of $\nu$, 
$\gamma$ and $\alpha$,
$(E_i-\nu^{-1}F_i)|_{d_i=0}$ could be negative,
then $\rho_{\text{NN}_{1\text{c}}}|_{d_i=0}$ would be greater than one since $\theta>0$. 
In other words, 
the convergence for small eigenvalues could be not good,
and cannot be fixed even with relaxation. 
For large eigenvalues $d_i$,
using~\eqref{eq:sobinf},
we find $E_i\sim_{\infty}2$ and $F_i\sim_{\infty}\frac1{d_i}$.
Thus,
we obtain $\lim_{d_i\rightarrow \infty} \rho_{\text{NN}_{1\text{c}}}=|1-2\theta|$,
which is independent of the interface $\alpha$. 
So large eigenvalue convergence is robust with relaxation, 
and one can get a good smoother using $\theta=1/2$.
Moreover,
we observe that algorithms NN$_{1\text{b}}$ and NN$_{1\text{c}}$ share similar behavior for large eigenvalues.
By equioscillating the convergence factor for small (i.e., $\rho_{\text{NN}_{1\text{c}}} |_{d_i=0}$) and large (i.e., $\rho_{\text{NN}_{1\text{c}}} |_{d_i\rightarrow\infty}$) eigenvalues, 
we obtain
\begin{equation}\label{eq:thetaNN1c}
	{\scriptstyle 
		\theta^*_{\text{NN}_{1\text{c}}} := \frac{2}{3+\tanh(\sqrt{\nu^{-1}}\alpha)\frac{ \gamma\sqrt{\nu^{-1}} + \tanh(\sqrt{\nu^{-1}}(T-\alpha)) }{\gamma\sqrt{\nu^{-1}}\tanh(\sqrt{\nu^{-1}}(T-\alpha)) + 1} -  \sqrt{\nu^{-1}}(\coth(\sqrt{\nu^{-1}}\alpha)  + \frac{\gamma\sqrt{\nu^{-1}} + \tanh(\sqrt{\nu^{-1}}(T-\alpha))}{1+ \gamma\sqrt{\nu^{-1}}\tanh(\sqrt{\nu^{-1}}(T-\alpha))})} 
	}.
\end{equation}
Note that when $(E_i-\nu^{-1}F_i)|_{d_i=0}<0$, 
the relaxation cannot improve the convergence for small eigenvalues,
thus,
\eqref{eq:thetaNN1c} could also be negative and cannot provide the optimal value of $\theta$ in this case.
One may use however a negative relaxation parameter $\theta$ to make the algorithm converge for small eigenvalues,
but this will induce divergence for large eigenvalues.
Based on the analysis, 
algorithm NN$_{1\text{c}}$ is a good smoother but not necessarily a good solver.

\begin{remark}
	One could also consider the update step~\eqref{eq:NN1btran} instead of~\eqref{eq:NN1ctran},
	and the convergence factor~\eqref{eq:rhoNN1c} will be $\max_{d_i\in\lambda(A)}|1-\theta d_i(F_i-\nu E_i)|$.
	For a similar reason as in Remark~\ref{rem:tran_NN1b},
	the algorithm diverges with this choice of update step.
\end{remark}

Together with the analysis in Section~\ref{sec:NN1b},
we observe that keeping the same update step~\eqref{eq:NN1atran} leads to divergent algorithms, 
when modifying the Neumann step. 
Thus, 
we should also adapt the update step according to the Neumann step.

\subsection{Category II}
We now study the algorithms in Category II which run the Dirichlet step only on the primal state $z_{i}$.

\subsubsection{Algorithm NN$_{2\text{a}}$}\label{sec:NN2a}
The most natural way is to correct $z_{i}$ by the primal correction state $\psi_{i}$.  
For $k=1,2,...$, 
algorithm NN$_{2\text{a}}$ first solves the Dirichlet step
\begin{equation}\label{eq:NN2}
\begin{aligned}
  &\left\{
    \begin{aligned}
      \begin{pmatrix}
        \dot z_{1,i}^k\\
        \dot \mu_{1,i}^k
      \end{pmatrix}
      +
      \begin{pmatrix}
        d_i & -\nu^{-1}\\
        -1 & -d_i
      \end{pmatrix}
      \begin{pmatrix}
        z_{1,i}^k\\
        \mu_{1,i}^k
      \end{pmatrix}
      &=
      \begin{pmatrix}
        0\\
        0
      \end{pmatrix} \text{ in } \Omega_1,\\
      z_{1,i}^k(0) &= 0,\\
      z_{1,i}^k(\alpha) & = f_{\alpha,i}^{k-1},
    \end{aligned}
  \right.\\
  &\left\{
    \begin{aligned}
      \begin{pmatrix}
        \dot z_{2,i}^k\\
        \dot \mu_{2,i}^k
      \end{pmatrix}
      +
      \begin{pmatrix}
        d_i & -\nu^{-1}\\
        -1 & -d_i
      \end{pmatrix}
      \begin{pmatrix}
        z_{2,i}^k\\
        \mu_{2,i}^k
      \end{pmatrix}
      &=
      \begin{pmatrix}
        0\\
        0
      \end{pmatrix} \text{ in } \Omega_2,\\
      z_{2,i}^k(\alpha) &= f_{\alpha,i}^{k-1},\\
      \mu_{2,i}^k(T) + \gamma z_{2,i}^k(T) &= 0,
    \end{aligned}
  \right.
   \end{aligned}
\end{equation}
then corrects the result by solving the Neumann step~\eqref{eq:NN1bcor}, 
and updates the transmission condition by~\eqref{eq:NN1btran} 

\begin{remark}\label{rem:update}
	Here, 
	it is more natural to consider the transmission condition only for $f_{\alpha,i}^k$. 
	This is due to the continuity of the primal state $z_{i}^k$ at the interface $\alpha$.
	In general, 
	we can show that an update step as~\eqref{eq:NN1ctran} will lead to divergence for a similar reason as in Remark~\ref{rem:tran_NN1b}.
	We can also show that a pair of transmission conditions $(f_{\alpha,i}^k, g_{\alpha,i}^k)$ will lead to non-convergent behavior (see Appendix~\ref{sec:app1}).
\end{remark}

For algorithm NN$_{2\text{a}}$, 
neither the Dirichlet~\eqref{eq:NN2} nor the Neumann step~\eqref{eq:NN1bcor} has the forward-backward structure in its current form. 
We have seen in Section~\ref{sec:NN1b} that we can recover this structure for the Neumann step~\eqref{eq:NN1bcor} which becomes a RN step.
Using the same idea, 
we can interpret 
$z_{1,i}^k(\alpha)=f_{\alpha,i}^{k-1}$ as $\dot\mu_{1,i}^k(\alpha) - d_i\mu_{1,i}^k(\alpha)=f_{\alpha,i}^{k-1}$ 
to recover the forward-backward structure,
and the Dirichlet step~\eqref{eq:NN2} then becomes a ND step.

For the convergence analysis, 
we transform the Dirichlet step~\eqref{eq:NN2} using~\eqref{eq:4id} and~\eqref{eq:sob}, 
and find 
\begin{equation}\label{eq:errNN2}
  \left\{
    \begin{aligned}
      \ddot z_{1,i}^k - \sigma_i^2 z_{1,i}^k &= 0 \text{ in } \Omega_1,\\
      z_{1,i}^k(0) &= 0, \\
      z_{1,i}^k(\alpha) &= f_{\alpha,i}^{k-1},
    \end{aligned}
  \right.
  \quad
  \left\{
    \begin{aligned}
      \ddot z_{2,i}^k - \sigma_i^2 z_{2,i}^k &= 0 \text{ in } \Omega_2,\\
      z_{2,i}^k(\alpha) &= f_{\alpha,i}^{k-1},\\
      \dot z_{2,i}^k(T) + \omega_iz_{2,i}^k(T)&= 0.
    \end{aligned}
  \right.
\end{equation}
The Neumann step becomes~\eqref{eq:errNN1bcor}, 
and we keep the same update step~\eqref{eq:NN1btran}. 
In particular, 
the convergence analysis also proceeds on a NN algorithm~\eqref{eq:errNN2}, \eqref{eq:errNN1bcor}, \eqref{eq:NN1btran}.
Using~\eqref{eq:errNNsol},
we can solve~\eqref{eq:errNN2} and determine the coefficients,
\begin{equation}\label{eq:D2AB}
	A_i^k = \frac{f_{\alpha,i}^{k-1}}{\sinh(a_i)}, \quad B_i^k= \frac{f_{\alpha,i}^{k-1}}{\sigma_i\cosh(b_i) + \omega_i\sinh(b_i)}.
\end{equation}
Combining them with~\eqref{eq:N2CD},
we update the transmission condition~\eqref{eq:NN1btran} and find
$f_{\alpha,i}^k=f_{\alpha,i}^{k-1} - \theta f_{\alpha,i}^{k-1}(E_i +F_i)$,
with 
\[E_i = \frac{\sigma_i\cosh(\sigma_i T)+\omega_i\sinh(\sigma_i T)}{(\sigma_i\sinh(b_i) + \omega_i\cosh(b_i))\sinh(a_i)}, F_i=\frac{\sigma_i\cosh(\sigma_i T)+\omega_i\sinh(\sigma_i T)}{(\sigma_i\cosh(b_i) + \omega_i\sinh(b_i))\cosh(a_i)}.\]
This leads to the following result.

\begin{theorem}
  Algorithm NN$_{2\text{a}}$~\eqref{eq:NN2}, \eqref{eq:NN1bcor}, \eqref{eq:NN1btran} converges if and only if
  \begin{equation}\label{eq:rhoNN2a}
  	\rho_{\text{NN}_{2\text{a}}} :=\max_{d_i\in\lambda(A)} | 1-\theta(E_i+F_i) |<1.
  \end{equation}
\end{theorem}

In particular, 
for a zero eigenvalue, 
substituting~\eqref{eq:sob0} into~\eqref{eq:rhoNN2a},
we have 
\begin{equation}\label{eq:rhoNN2ad0}
	\begin{aligned}
		\rho_{\text{NN}_{2\text{a}}}|_{d_i=0} = \Big|1-\theta \Big(2 +  &\coth(\sqrt{\nu^{-1}}\alpha)\frac{\coth\big(\sqrt{\nu^{-1}}(T-\alpha)\big) + \gamma\sqrt{\nu^{-1}}}{1 + \gamma\sqrt{\nu^{-1}}\coth\big(\sqrt{\nu^{-1}}(T-\alpha)\big)} \\
		&+ \tanh(\sqrt{\nu^{-1}}\alpha)\frac{\tanh\big(\sqrt{\nu^{-1}}(T-\alpha)\big) + \gamma\sqrt{\nu^{-1}}}{1+ \gamma\sqrt{\nu^{-1}}\tanh\big(\sqrt{\nu^{-1}}(T-\alpha)\big)}\Big) \Big|.
	\end{aligned}
\end{equation}
For large eigenvalues $d_i$, 
using~\eqref{eq:sobinf},
we find $E_i\sim_{\infty}2$ and $F_i\sim_{\infty}2$.
Thus, 
we obtain $\lim_{d_i\rightarrow \infty} \rho_{\text{NN}_{2\text{a}}}=|1-4\theta|$,
which is independent of the interface $\alpha$. 
So the convergence for high frequencies is robust with relaxation, 
and one can get a good smoother using $\theta = 1/4$.
By equioscillating the convergence factor for small (i.e., $\rho_{\text{NN}_{2\text{a}}}|_{d_i=0}$) and large (i.e., $\rho_{\text{NN}_{2\text{a}}}|_{d_i\rightarrow\infty}$) eigenvalues, 
we obtain the relaxation parameter
\begin{equation}\label{eq:thetaNN2a}
	{\scriptstyle 
		\theta_{\text{NN}_{2\text{a}}}^*:= \frac{2}{6+ \coth(\sqrt{\nu^{-1}}\alpha)\frac{\coth(\sqrt{\nu^{-1}}(T-\alpha)) + \gamma\sqrt{\nu^{-1}}}{1 + \gamma\sqrt{\nu^{-1}}\coth(\sqrt{\nu^{-1}}(T-\alpha))} + \tanh(\sqrt{\nu^{-1}}\alpha)\frac{\tanh(\sqrt{\nu^{-1}}(T-\alpha)) + \gamma\sqrt{\nu^{-1}}}{1 + \gamma\sqrt{\nu^{-1}}\tanh(\sqrt{\nu^{-1}}(T-\alpha))}}
	},
\end{equation}
which is smaller than 1/3.
In the case with no final state, 
i.e., $\gamma=0$, 
we have 
\[\theta_{\text{NN}_{2\text{a}}}^*|_{\gamma=0}= \frac{2}{6+ \coth(\sqrt{\nu^{-1}}\alpha)\coth(\sqrt{\nu^{-1}}(T-\alpha)) + \tanh(\sqrt{\nu^{-1}}\alpha)\tanh(\sqrt{\nu^{-1}}(T-\alpha))}.\]
Using properties of the hyperbolic tangent and cotangent, 
we find 
\[\begin{aligned}
\coth(\sqrt{\nu^{-1}}\alpha)\coth(\sqrt{\nu^{-1}}(T-\alpha)) + \tanh(\sqrt{\nu^{-1}}\alpha)\tanh(\sqrt{\nu^{-1}}(T-\alpha))\geq \\
\coth^2(\sqrt{\nu^{-1}}\frac T2) + \tanh^2(\sqrt{\nu^{-1}}\frac T2) > 2,
\end{aligned}\]
thus $\theta_{\text{NN}_{2\text{a}}}^*< \frac 14$.
Based on the analysis,
algorithm NN$_{2\text{a}}$ is a good smoother and can also be a good solver.
However, 
it is less clear under what condition $\theta_{\text{NN}_{2\text{a}}}^*$ is the optimal relaxation parameter, 
since the monotonicity of the convergence factor with respect to the eigenvalues $d_i$ is not clear even in the case $\gamma=0$.
This has been observed in our numerical experiments.

\subsubsection{Algorithm NN$_{2\text{b}}$}\label{sec:NN2b}
We can also keep the Dirichlet step~\eqref{eq:NN2}, 
but apply the Neumann step only to the dual correction state $\phi_i$ as in~\eqref{eq:NN1ccor}. 
As for the update step, 
we first consider to take the same update as for algorithm NN$_{2\text{a}}$, 
i.e., \eqref{eq:NN1btran}.

For the convergence analysis, 
we actually solve a DD step~\eqref{eq:errNN2} and correct by a RR step~\eqref{eq:errNN1ccor}. 
Using~\eqref{eq:D2AB} and~\eqref{eq:N3CD}, 
we update the transmission condition~\eqref{eq:NN1btran} and find 
$f_{\alpha,i}^k=f_{\alpha,i}^{k-1} (1- \theta d_i(F_i-\nu E_i))$ with 
\[\begin{aligned}
E_i&= \frac{\sigma_i\cosh(\sigma_iT)+\omega_i\sinh(\sigma_iT) }{\sigma_i\gamma\sinh(b_i) + \beta_i\cosh(b_i)}\frac1{\sinh(a_i)},\\
F_i&=\frac{\sigma_i\cosh(\sigma_iT)+\omega\sinh(\sigma_iT)}{(\sigma_i\cosh(b_i) + \omega_i\sinh(b_i)) (\sigma_i \sinh(a_i)+d_i\cosh(a_i))}.
\end{aligned}\] 
We then obtain the convergence factor 
\begin{equation}\label{eq:rhoNN2b}
  	\rho_{\text{NN}_{2\text{b}}} :=\max_{d_i\in\lambda(A)} | 1- \theta d_i(F_i-\nu E_i) |<1.
\end{equation}

To get more insight,
we first study the extremal cases.
For a zero eigenvalue, 
$d_i=0$, 
substituting~\eqref{eq:sob0} into~\eqref{eq:rhoNN2b},
we have $(F_i-\nu E_i)|_{d_i=0}=0$.
Hence, 
we find $\rho_{\text{NN}_{2\text{b}}}|_{d_i=0} = 1$,
which is independent of the relaxation parameter.
In other words,
the convergence behavior of algorithm NN$_{2\text{b}}$ is not good for small eigenvalues, 
and the relaxation cannot fix this problem.
For large eigenvalues $d_i$, 
using~\eqref{eq:sobinf},
we find $E_i\sim_{\infty}4d_i$ and $F_i\sim_{\infty}\frac1{d_i}$.
Thus, 
we obtain $1- \theta d_i(F_i-\nu E_i)\sim_{\infty} 4\nu \theta d_i^2$ and $\lim_{d_i\rightarrow \infty} \rho_{\text{NN}_{2\text{b}}}=\infty$,
which is divergent,
and cannot be fixed with relaxation.
Generally,
we have the following result.

\begin{theorem}\label{thm:NN2b}
  Algorithm NN$_{2\text{b}}$~\eqref{eq:NN2} \eqref{eq:NN1ccor} \eqref{eq:NN1btran} always diverges.
\end{theorem}

\begin{proof}
	Using the formula of $E_i$ and $F_i$, 
	we find
	$F_i-\nu E_i = \frac{-\nu d_i}{\sigma_i \sinh(a_i)+d_i\cosh(a_i)}$ $\frac{(\sigma_i\cosh(\sigma_iT)+\omega_i\sinh(\sigma_iT))^2}{\sinh(a_i)(\sigma_i\gamma\sinh(b_i) + \beta_i\cosh(b_i))(\sigma_i\cosh(b_i) + \omega_i\sinh(b_i))}$
	which is negative or zero (if $d_i=0$).
	Since $\theta$ and $\nu$ are both positive, 
	$1- \theta d_i(F_i-\nu E_i)\geq 1$ which concludes the proof.
\end{proof}

The above result shows that algorithm NN$_{2\text{b}}$ diverges with a positive relaxation parameter $\theta$.
Moreover,
this divergence cannot be fixed even with a negative $\theta$,
since the convergence factor is one for a zero eigenvalue,
and is equivalent to $4\nu |\theta| d_i^2$ for large eigenvalues.
In general, 
algorithm NN$_{2\text{b}}$ is neither a good smoother nor a good solver.

\begin{remark}\label{rem:NN2b}
	Compared with algorithm NN$_{2\text{a}}$, 
	we change the Neumann step but keep the same update step.
	One can also consider the update step~\eqref{eq:NN1ctran},
	since the Neumann correction~\eqref{eq:NN1ccor} is only applied to the dual correction state $\phi_i$.
	Following the same computation,
	the convergence factor~\eqref{eq:rhoNN2b} then becomes
\[\max_{d_i\in\lambda(A)} | 1- \theta (E_i-\nu^{-1} F_i) |,\]
with $E_i-\nu^{-1} F_i\geq0$.
	However,
	this does not change the poor convergence behavior for both small and large eigenvalues.
	Indeed, 
	we still have $(E_i-\nu^{-1} F_i)|_{d_i=0}=0$,
	hence $\rho_{\text{NN}_{2\text{b}}}|_{d_i=0} = 1$,
	and $\lim_{d_i\rightarrow \infty} \rho_{\text{NN}_{2\text{b}}}=\infty$.
	Thus,
	the modified algorithm stays divergent.
	Furthermore,
	for a similar reason as mentioned in Appendix~\ref{sec:app1},
	the algorithm is also divergent when considering the update step~\eqref{eq:NN1atran} with a pair of transmission conditions $(f_{\alpha,i}^k, g_{\alpha,i}^k)$. 
\end{remark}

Based on the analysis,
we cannot find a good NN algorithm when combining the Dirichlet step~\eqref{eq:NN2} with the Neumann step~\eqref{eq:NN1ccor}.

\subsubsection{Algorithm NN$_{2\text{c}}$}\label{sec:NN2c}
If we apply the correction to the pair $(\psi_i,\phi_i)$,
then the Neumann step immediately has the forward-backward structure. 
In this way, 
algorithm NN$_{2\text{c}}$ solves first the Dirichlet step~\eqref{eq:NN2}, 
next the Neumann step~\eqref{eq:NN1acor} and updates the transmission condition by~\eqref{eq:NN1btran}.

For the convergence analysis, 
we solve a DD step~\eqref{eq:errNN2} followed by a RN correction step~\eqref{eq:errNN1acor}. 
Using~\eqref{eq:D2AB} and~\eqref{eq:N1CD}, 
we update the transmission condition~\eqref{eq:NN1btran} and find 
$f_{\alpha,i}^k=f_{\alpha,i}^{k-1} (1- \theta (E_i + d_iF_i))$ with 
\[\begin{aligned}
E_i&= \frac{\sigma_i\cosh(\sigma_iT)+\omega_i\sinh(\sigma_iT)}{(\sigma_i\sinh(b_i) + \omega_i\cosh(b_i))\sinh(a_i)},\\
F_i&=\frac{\sigma_i\cosh(\sigma_iT)+\omega_i\sinh(\sigma_iT)}{(\sigma_i\cosh(b_i) + \omega_i\sinh(b_i))(\sigma_i \sinh(a_i)+d_i\cosh(a_i))}.
\end{aligned}\] 
We then obtain the following result. 

\begin{theorem}
  Algorithm NN$_{2\text{c}}$~\eqref{eq:NN2}, \eqref{eq:NN1acor}, \eqref{eq:NN1btran} converges if and only if
  \begin{equation}\label{eq:rhoNN2c}
  	\rho_{\text{NN}_{2\text{c}}} :=\max_{d_i\in\lambda(A)} | 1- \theta (E_i+d_i F_i) |<1.
   \end{equation}
\end{theorem}

For a zero eigenvalue $d_i=0$, 
substituting the identities~\eqref{eq:sob0} into~\eqref{eq:rhoNN2c},
we find 
\begin{equation}\label{eq:rhoNN2cd0}
	\rho_{\text{NN}_{2\text{c}}}|_{d_i=0} = \Big|1-\theta\big(1+\coth(\sqrt{\nu^{-1}}\alpha)\frac{\coth(\sqrt{\nu^{-1}}(T-\alpha))+\gamma\sqrt{\nu^{-1}}}{1+\gamma\sqrt{\nu^{-1}}\coth(\sqrt{\nu^{-1}}(T-\alpha))}\big)\Big|.
\end{equation}
For large eigenvalues $d_i$, 
using~\eqref{eq:sobinf},
we find $E_i\sim_{\infty}2$ and $F_i\sim_{\infty}\frac1{d_i}$.
Thus, 
we obtain $\lim_{d_i\rightarrow \infty} \rho_{\text{NN}_{2\text{c}}}=|1-3\theta|$,
which is independent of the interface $\alpha$. 
So the convergence for high frequencies is robust with relaxation, 
and one can get a good smoother using $\theta = 1/3$.
By equioscillating the convergence factor for small (i.e., $\rho_{\text{NN}_{2\text{c}}} |_{d_i=0}$) and large (i.e., $\rho_{\text{NN}_{2\text{c}}} |_{d_i\rightarrow\infty}$) eigenvalues, 
we obtain
\begin{equation}\label{eq:thetaNN2c}
	\theta^*_{\text{NN}_{2\text{c}}} := \frac{2}{4+\coth(\sqrt{\nu^{-1}}\alpha)\frac{\coth(\sqrt{\nu^{-1}}(T-\alpha))+\gamma\sqrt{\nu^{-1}}}{1+\gamma\sqrt{\nu^{-1}}\coth(\sqrt{\nu^{-1}}(T-\alpha))}},
\end{equation}
which is smaller than 1/2. 
In the case $\gamma=0$, 
the relaxation parameter $\theta^*_{\text{NN}_{2\text{c}}}$ is bounded by 2/5.
However, 
it is also not clear under what condition $\theta^*_{\text{NN}_{2\text{c}}}$ is the optimal relaxation parameter,
since the monotonicity of $E_i+d_i F_i$ with respect to $d_i$ is less clear, and depends on the parameter values $\alpha$, $\gamma$ and $\nu$.
Generally,
algorithm NN$_{2\text{c}}$ is both a good smoother and a good solver with a well-chosen $\theta$.

\begin{remark}\label{rem:NN2c}
	Instead of choosing~\eqref{eq:NN1btran} as the update step,
	one could have considered the update step~\eqref{eq:NN1ctran}.
	Following the same computation,
	the convergence factor becomes $\max_{d_i\in\lambda(A)} | 1- \theta (d_iE_i-\nu^{-1} F_i) |$,
	which diverges for large eigenvalues.
	Furthermore,
	the algorithm will also be divergent when considering the update step~\eqref{eq:NN1atran} with a pair transmission conditions $(f_{\alpha,i}^k, g_{\alpha,i}^k)$ as mentioned in Appendix~\ref{sec:app1}.
\end{remark}

\subsection{Category III}
The algorithms in Category III run the Dirichlet step only on the dual state $\mu_{i}$, 
and according to the Neumann step, 
there are three variants.

\subsubsection{Algorithm NN$_{3\text{a}}$}
As in Section~\ref{sec:NN2a}, 
the most natural way is to correct the dual state $\mu_{i}$ only by the dual correction state $\phi_{i}$. 
In this way,
for $k=1,2,...$, 
algorithm NN$_{3\text{a}}$ first solves the Dirichlet step
\begin{equation}\label{eq:NN3}
\begin{aligned}
  &\left\{
    \begin{aligned}
      \begin{pmatrix}
        \dot z_{1,i}^k\\
        \dot \mu_{1,i}^k
      \end{pmatrix}
      +
      \begin{pmatrix}
        d_i & -\nu^{-1}\\
        -1 & -d_i
      \end{pmatrix}
      \begin{pmatrix}
        z_{1,i}^k\\
        \mu_{1,i}^k
      \end{pmatrix}
      &=
      \begin{pmatrix}
        0\\
        0
      \end{pmatrix} \text{ in } \Omega_1,\\
      z_{1,i}^k(0) &= 0,\\
      \mu_{1,i}^k(\alpha) & = f_{\alpha,i}^{k-1},
    \end{aligned}
  \right.\\
  &\left\{
    \begin{aligned}
      \begin{pmatrix}
        \dot z_{2,i}^k\\
        \dot \mu_{2,i}^k
      \end{pmatrix}
      +
      \begin{pmatrix}
        d_i & -\nu^{-1}\\
        -1 & -d_i
      \end{pmatrix}
      \begin{pmatrix}
        z_{2,i}^k\\
        \mu_{2,i}^k
      \end{pmatrix}
      &=
      \begin{pmatrix}
        0\\
        0
      \end{pmatrix} \text{ in } \Omega_2,\\
      \mu_{2,i}^k(\alpha) &= f_{\alpha,i}^{k-1},\\
      \mu_{2,i}^k(T) + \gamma z_{2,i}^k(T) &= 0,
    \end{aligned}
  \right.
   \end{aligned}
\end{equation}
then corrects the above result by solving the Neumann step~\eqref{eq:NN1ccor}, 
and updates the transmission condition by~\eqref{eq:NN1ctran}. 

Similar to Remark~\ref{rem:update}, 
we choose here the update step~\eqref{eq:NN1ctran} because of the continuity of the dual state $\mu_i^k$ at the interface $\alpha$,
since other choices of the update step will induce divergence behavior.
Regarding the forward-backward structure for the Dirichlet step~\eqref{eq:NN3}, 
we can recover it by interpreting $\mu_{2,i}^k(\alpha) = f_{\alpha,i}^{k-1}$ as $\dot z_{2,i}^k(\alpha) + d_i z_{2,i}^k(\alpha) = f_{\alpha,i}^{k-1}$.
The Dirichlet step~\eqref{eq:NN3} then becomes a NR step.

To analyze algorithm NN$_{3\text{a}}$, 
we can rewrite the Dirichlet step~\eqref{eq:NN3} using~\eqref{eq:4id} and~\eqref{eq:sob}, 
and find
\begin{equation}\label{eq:errNN3}
 \left\{
    \begin{aligned}
      \ddot z_{1,i}^k - \sigma_i^2 z_{1,i}^k &= 0 \text{ in } \Omega_1,\\
      z_{1,i}^k(0) &= 0, \\
      \dot z_{1,i}^k(\alpha) + d_i z_{1,i}^k(\alpha) &= f_{\alpha,i}^{k-1},
    \end{aligned}
  \right.
  \quad
  \left\{
    \begin{aligned}
      \ddot z_{2,i}^k - \sigma_i^2 z_{2,i}^k &= 0 \text{ in } \Omega_2,\\
      \dot z_{2,i}^k(\alpha) + d_i z_{2,i}^k(\alpha) &= f_{\alpha,i}^{k-1},\\
      \dot z_{2,i}^k(T) + \omega_iz_{2,i}^k(T)&= 0.
    \end{aligned}
  \right.
\end{equation}
We then correct the above RR step by a RR correction~\eqref{eq:errNN1ccor}, 
which is also the equivalent of the Neumann step~\eqref{eq:NN1ccor}.
And the update step~\eqref{eq:NN1ctran} becomes
\begin{equation}\label{eq:NN3atran}
    f_{\alpha,i}^k = f_{\alpha,i}^{k-1} - \theta \big(\dot \psi_{1,i}^k(\alpha) + d_i \psi_{1,i}^k(\alpha)+\dot \psi_{2,i}^k(\alpha) + d_i \psi_{2,i}^k(\alpha)\big).
\end{equation} 
Using~\eqref{eq:errNNsol}, 
we can solve explicitly~\eqref{eq:errNN3} and determine the coefficients
\begin{equation}\label{eq:D3AB}
	A_i^k = \frac{f_{\alpha,i}^{k-1}}{\sigma_i \cosh(a_i) + d_i \sinh(a_i)}, \quad B_i^k = -\nu\frac{f_{\alpha,i}^{k-1}}{\sigma_i\gamma\cosh(b_i) + \beta_i\sinh(b_i)}.
\end{equation}
Combining with~\eqref{eq:N3CD}, 
we update the transmission condition~\eqref{eq:NN3atran} and obtain
$f_{\alpha,i}^k = f_{\alpha,i}^{k-1}- \theta f_{\alpha,i}^{k-1}(E_i+F_i)$
with 
\[\begin{aligned}
E_i= \frac{\sigma_i\cosh(\sigma_iT) + \omega_i\sinh(\sigma_iT)}{\sigma_i\gamma\sinh(b_i) + \beta_i\cosh(b_i)}\frac1{\sigma_i \cosh(a_i) + d_i \sinh(a_i)},\\
F_i=\frac{\sigma_i\cosh(\sigma_iT)+\omega\sinh(\sigma_iT)}{\sigma_i\gamma\cosh(b_i) + \beta_i\sinh(b_i)}\frac1{\sigma_i \sinh(a_i)+d_i\cosh(a_i)}.
\end{aligned}\] 
Thus, 
we have the following result.

\begin{theorem}
  Algorithm NN$_{3\text{a}}$~\eqref{eq:NN3}, \eqref{eq:NN1ccor}, \eqref{eq:NN1ctran} converges if and only if
  \begin{equation}\label{eq:rhoNN3a}
  	\rho_{\text{NN}_{3\text{a}}} :=\max_{d_i\in\lambda(A)} | 1- \theta (E_i+F_i) |<1.
  \end{equation}
\end{theorem}

We consider some special cases to get more insight in the convergence factor~\eqref{eq:rhoNN3a}.
Assuming no final target (i.e., $\gamma=0$) and a symmetric decomposition $\alpha=\frac T2$ (i.e., $a_i=b_i$),
we find that $E_i$ and $F_i$ are actually the same as for algorithm NN$_{2\text{a}}$ in Section~\ref{sec:NN2a}.
Hence, 
the convergence factor~\eqref{eq:rhoNN3a} is as~\eqref{eq:rhoNN2a} under this assumption, 
and NN$_{2\text{a}}$ and NN$_{3\text{a}}$ are actually the same algorithm.
Moreover, 
for a zero eigenvalue, 
substituting~\eqref{eq:sob0} into~\eqref{eq:rhoNN3a},
we find exactly the same formula as~\eqref{eq:rhoNN2ad0}.
Thus, 
the two algorithms NN$_{2\text{a}}$ and NN$_{3\text{a}}$ share the same behavior for small eigenvalues.
On the other hand,
using~\eqref{eq:sobinf} for large eigenvalues $d_i$,
we find $E_i\sim_{\infty} 2$ and $F_i\sim_{\infty} 2$.
This implies that $\lim_{d_i\rightarrow \infty} \rho_{\text{NN}_{3\text{a}}} = |1-4\theta|$,
which is the same as for algorithm NN$_{2\text{a}}$.
Once again, 
the two algorithms NN$_{2\text{a}}$ and NN$_{3\text{a}}$ share the same behavior for large eigenvalues.
Hence,
we obtain the same relaxation parameter $\theta_{\text{NN}_{3\text{a}}}^*=\theta_{\text{NN}_{2\text{a}}}^*$ as defined in~\eqref{eq:thetaNN2a}.
In general, 
algorithm NN$_{3\text{a}}$ seems to be very similar to NN$_{2\text{a}}$, 
and we could also expect it to be a good smoother and solver.

\subsubsection{Algorithm NN$_{3\text{b}}$}
The second variant in Category III consists in applying the Neumann step to the primal correction state $\psi_i$. 
In this way, 
we consider the algorithm that first solves the Dirichlet step~\eqref{eq:NN3}, 
followed by the Neumann step~\eqref{eq:NN1bcor},
and updates the transmission condition by~\eqref{eq:NN1ctran}.

For the convergence analysis, 
we solve a RR step~\eqref{eq:errNN3} and correct by a NN step~\eqref{eq:errNN1bcor}. 
Using~\eqref{eq:D3AB} and~\eqref{eq:N2CD},
we can update the transmission condition~\eqref{eq:NN3atran} and find
$f_{\alpha,i}^k = f_{\alpha,i}^{k-1}-f_{\alpha,i}^{k-1}\theta d_i(E_i-\nu F_i)$
with 
\[\begin{aligned}
E_i&= \frac{\sigma_i\cosh(\sigma_iT)+\omega_i\sinh(\sigma_iT)}{(\sigma_i\sinh(b_i)+ \omega_i\cosh(b_i))(\sigma_i \cosh(a_i) + d_i \sinh(a_i))},\\
F_i&=\frac{\sigma_i\cosh(\sigma_iT)+\omega\sinh(\sigma_iT)}{(\sigma_i\gamma\cosh(b_i)+\beta_i\sinh(b_i))\cosh(a_i)}. 
\end{aligned}\] 
This leads to the convergence factor
\begin{equation}\label{eq:rhoNN3b}
  	\rho_{\text{NN}_{3\text{b}}} :=\max_{d_i\in\lambda(A)} | 1- \theta d_i(E_i-\nu F_i) |<1.
  \end{equation}

We first study the extreme cases.
For a zero eigenvalue, 
substituting the identities~\eqref{eq:sob0} into~\eqref{eq:rhoNN3b},
we find $(E_i-\nu F_i)|_{d_i=0} = 0$,
and hence $\rho_{\text{NN}_{3\text{b}}}|_{d_i=0} = 1$.
This is once again independent of the relaxation parameter.
In other words,
the convergence of this algorithm is not good for small eigenvalues, 
and the relaxation cannot fix this problem.
For large eigenvalues $d_i$, 
using~\eqref{eq:sobinf},
we find $E_i\sim_{\infty}\frac1{d_i}$ and $F_i\sim_{\infty}4d_i$.
Thus, 
we obtain $\rho_{\text{NN}_{3\text{b}}}\sim_{\infty} 4\nu \theta d_i^2$ and $\lim_{d_i\rightarrow \infty} \rho_{\text{NN}_{3\text{b}}}=\infty$,
which is divergent and cannot be fixed with relaxation.
In general,
we have the following result.

\begin{theorem}
  Algorithm NN$_{3\text{b}}$~\eqref{eq:NN3}, \eqref{eq:NN1bcor}, \eqref{eq:NN1ctran} always diverges.
\end{theorem}

\begin{proof}
	Following the same idea as in the proof of Theorem~\ref{thm:NN2b},
	we can show that $E_i-\nu F_i$ is always negative or zero, 
	and this concludes the proof.  
\end{proof}

\begin{remark}
	One could have also applied a similar strategy as in Remark~\ref{rem:NN2b},
	that is, 
	considering the update step~\eqref{eq:NN1btran} instead of~\eqref{eq:NN1ctran}.
	The convergence factor~\eqref{eq:rhoNN3b} then becomes
	$\max_{d_i\in\lambda(A)} | 1- \theta (E_i-\nu F_i) |$.
	Once again,
	this does not change the poor convergence behavior for both small and large eigenvalues.
\end{remark}

Similar to algorithm NN$_{2\text{b}}$, 
algorithm NN$_{3\text{b}}$ is neither a good smoother nor a good solver,
and other choices of the update step will not change this.
Together with Section~\ref{sec:NN2b},
we observe that, 
applying the Dirichlet step to the primal state $z_i$ (resp. dual state $\mu_i$) and correcting the result by a Neumann step to the dual correction state $\phi_i$ (resp. primal correction state $\psi_i$),
will lead to divergent algorithms,
and cannot be fixed even by adapting the update step.

\subsubsection{Algorithm NN$_{3\text{c}}$}
The last variant consists in applying the Neumann step to the pair $(\psi_i,\phi_i)$. 
In this way, 
the NN$_{3\text{b}}$ algorithm solves first the Dirichlet step~\eqref{eq:NN3}, 
next the Neumann step~\eqref{eq:NN1acor} which also has the forward-backward structure.
Then it updates the transmission condition by~\eqref{eq:NN1ctran}.

For the convergence analysis, 
we solve a RR step~\eqref{eq:errNN3} followed by a NR correction~\eqref{eq:errNN1acor}. 
Using~\eqref{eq:D3AB} and~\eqref{eq:N1CD}, 
we update the transmission condition~\eqref{eq:NN3atran} and find
$f_{\alpha,i}^k = f_{\alpha,i}^{k-1} (1- \theta(d_iE_i+F_i))$
with 
\[\begin{aligned}
E_i= \frac{\sigma_i\cosh(\sigma_iT)+\omega_i\sinh(\sigma_iT)}{(\sigma_i\sinh(b_i)+\omega_i\cosh(b_i))(\sigma_i \cosh(a_i) + d_i \sinh(a_i))},\\
F_i=\frac{\sigma_i\cosh(\sigma_iT)+\omega\sinh(\sigma_iT)}{(\sigma_i\gamma\cosh(b_i)+\beta_i\sinh(b_i))(\sigma_i \sinh(a_i)+d_i\cosh(a_i))}.
\end{aligned}\] 
We thus find the following result.

\begin{theorem}
	Algorithm NN$_{3\text{c}}$~\eqref{eq:NN3}, \eqref{eq:NN1acor}, \eqref{eq:NN1ctran} converges if and only if
	\begin{equation}\label{eq:rhoNN3c}
  		\rho_{\text{NN}_{3\text{c}}} :=\max_{d_i\in\lambda(A)} | 1- \theta (d_iE_i+F_i) |<1.
	\end{equation}
\end{theorem}

We consider some special cases to get more insight.
Assuming no final target (i.e., $\gamma=0$) and a symmetric decomposition $\alpha=\frac T2$ (i.e., $a_i=b_i$),
we find that $E_i$ is actually the same as the $F_i$ for algorithm NN$_{2\text{c}}$, 
and $F_i$ is the same as the $E_i$ for algorithm NN$_{2\text{c}}$ in Section~\ref{sec:NN2c}.
Hence, 
NN$_{2\text{c}}$ and NN$_{3\text{c}}$ are the same algorithm under this assumption.
For a zero eigenvalue, 
$d_i=0$, 
substituting the identities~\eqref{eq:sob0} into~\eqref{eq:rhoNN3c},
we find 
$\rho_{\text{NN}_{3\text{c}}}|_{d_i=0} =\rho_{\text{NN}_{2\text{c}}}|_{d_i=0}$ as in~\eqref{eq:rhoNN2c}.
In other words,
algorithms NN$_{2\text{c}}$ and NN$_{3\text{c}}$ have a similar behavior for small eigenvalues.
For large eigenvalues $d_i$, 
using~\eqref{eq:sobinf},
we find $E_i\sim_{\infty}\frac1{d_i}$ and $F_i\sim_{\infty}2$.
Thus, 
we obtain $\lim_{d_i\rightarrow \infty} \rho_{\text{NN}_{3\text{c}}}=|1-3\theta|$,
which is independent of the interface $\alpha$. 
So the convergence for large eigenvalues is robust with relaxation, 
and one can get a good smoother using $\theta = 1/3$.
Furthermore,
we find again similar behavior between algorithms NN$_{2\text{c}}$ and NN$_{3\text{c}}$ for large eigenvalues.
Using hence equioscillation, 
we obtain $\theta^*_{\text{NN}_{3\text{c}}}=\theta^*_{\text{NN}_{2\text{c}}}$ as defined in~\eqref{eq:thetaNN2c}.
Based on all these similarities with algorithm NN$_{2\text{c}}$,
algorithm NN$_{3\text{c}}$ is also a good smoother and solver.
Also for a similar reason as explained in Remark~\ref{rem:NN2c},
other choices of the update step will lead to divergent behavior.

\section{Numerical results}\label{sec:test}
We illustrate now our nine new time domain decomposition algorithms with numerical experiments. 
As mentioned in the convergence analysis,
some algorithms are much more sensitive to the chosen parameters than others.
To well illustrate and compare these algorithms,
we consider two different test cases,
\begin{itemize}
	\item[\textbf{case A:}] The time interval $\Omega = (0,1)$ is subdivided into $\Omega_1=(0,0.5)$, $\Omega_2=(0.5,1)$ (i.e., symmetric),
	and the objective function has no explicit final target term ($\gamma= 0$). 
	The regularization parameter is $\nu = 0.1$.
	\item[\textbf{case B:}] The time interval $\Omega = (0,5)$ is subdivided into $\Omega_1=(0,1)$, $\Omega_2=(1,5)$ (i.e., asymmetric),
	and the objective function has a final target term with $\gamma= 10$. 
	The regularization parameter is $\nu = 10$.
\end{itemize} 
For each test,
we will investigate the performance by plotting the convergence factor as a function of the eigenvalues $d_i \in [10^{-2}, 10^2 ]$.

\subsection{Convergence factor of NN$_{2\text{b}}$ and NN$_{3\text{b}}$}
We first illustrate the behavior of NN$_{2\text{b}}$ and NN$_{3\text{b}}$ separately, 
since their convergence analyses are very similar, 
and both algorithms are divergent.
Figure~\ref{fig:div_NN2b_NN3b} shows the behavior of the convergence factor as a function of the eigenvalues for these two algorithms.
\begin{figure}
	\centering
	\includegraphics[scale=0.25]{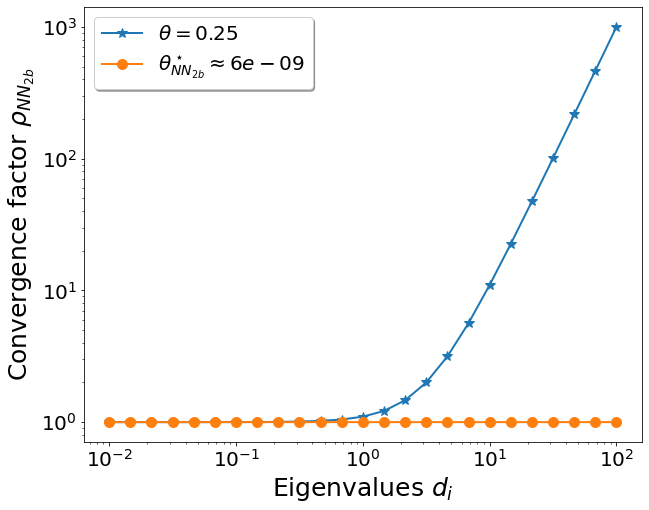}
	\includegraphics[scale=0.25]{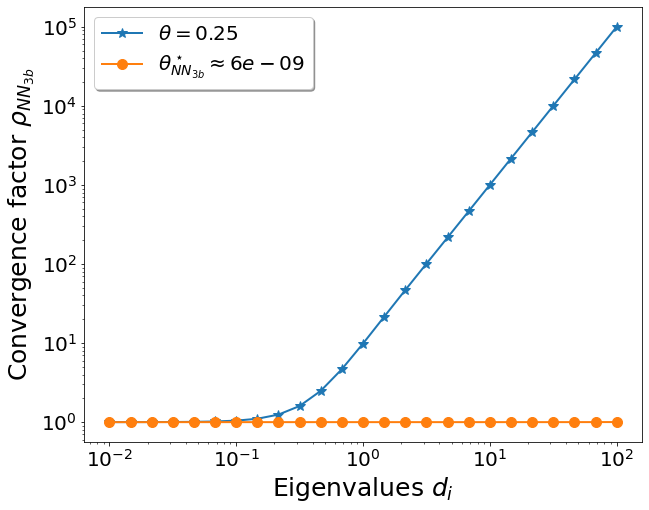}
	\caption{Convergence factor with $\theta=0.25$ of NN$_{2\text{b}}$ and NN$_{3\text{b}}$ as a function of the eigenvalues $d_i \in [10^{-2}, 10^2 ]$.
	Left: case A for NN$_{2\text{b}}$.
	Right: case B for NN$_{3\text{b}}$.}
	\label{fig:div_NN2b_NN3b}
\end{figure}
More precisely,
both algorithms diverge in the case $\theta=0.25$.
And for both test cases A and B,
the two algorithms diverge violently for large eigenvalues with the scale of 10$^3$ for NN$_{2\text{b}}$ and 10$^5$ for NN$_{3\text{b}}$. 
This corresponds to our estimate $4\nu\theta d_i^2$.
By applying optimization\footnote{We use in this paper the optimization toolbox \textit{scipy.optimize.fmin} in python.},
we find the optimal relaxation parameter is approximately zero for both algorithms in the test cases.
As shown in our analysis,
the best one can do is to choose $\theta=0$ to compensate the bad large eigenvalue behavior,
yet the algorithms are still divergent.
Note that NN$_{2\text{b}}$ and NN$_{3\text{b}}$ in the case $\theta=0$ are actually a classical Schwarz type algorithm,
which does not converge without overlap.
Therefore,
NN$_{2\text{b}}$ and NN$_{3\text{b}}$ are not good algorithms and cannot be improved with relaxation.

\subsection{Convergence factor of NN$_{1\text{a}}$ with different $\theta$}
The second test is dedicated to the most natural Neumann-Neumann algorithm NN$_{1\text{a}}$.
Based on our analysis, 
NN$_{1\text{a}}$ is only a good smoother but not a good solver.
Therefore,
we choose some different relaxation parameters $\theta$ and show the behavior of the convergence factor as a function of the eigenvalues in Figure~\ref{fig:NN1a}.
\begin{figure}
\centering
\includegraphics[scale=0.25]{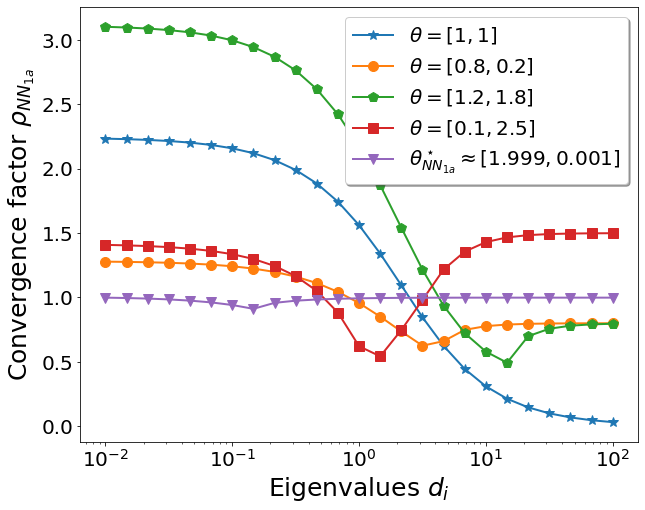}
\includegraphics[scale=0.25]{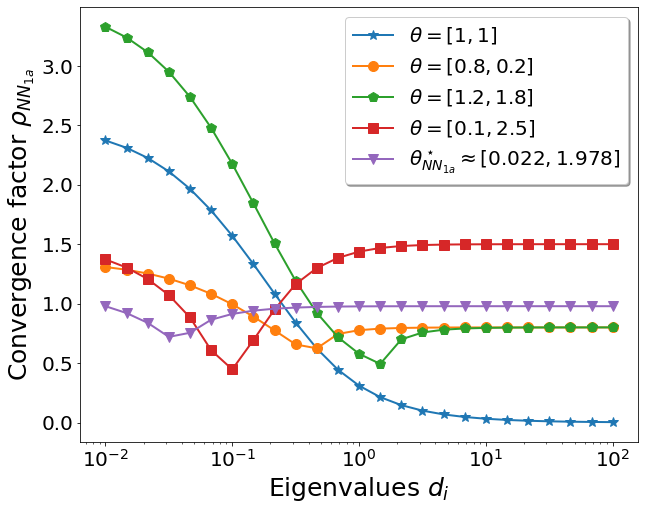}
\caption{Convergence factor with different relaxation parameters $\theta$ of NN$_{1\text{a}}$ as a function of the eigenvalues $d_i \in [10^{-2}, 10^2 ]$.
	Left: case A.
	Right: case B.}
\label{fig:NN1a}
\end{figure}
For both test cases A and B,
NN$_{1\text{a}}$ has similar behavior for the tested parameters $\theta$.
In the case $\theta=[0.8,0.2]$ and $\theta=[1.2,1.8]$,
the convergence behavior is the same for large eigenvalues.
Indeed, 
our analysis shows that $\lim_{d_i\rightarrow\infty} \rho_{\text{NN}_{1\text{a}}} = \{|1-\theta_1|,|1-\theta_2|\}$,
and in this case equals to 0.8 for both $\theta$.
Furthermore,
we observe that NN$_{1\text{a}}$ is a good smoother with the choice $\theta=[1,1]$. 
By using optimization,
we find that the optimal relaxation parameter has the form that one goes to zero and the other one goes to two, 
yet with a poor convergence.
Therefore,
NN$_{1\text{a}}$ can be a good smoother but not a good solver.

\subsection{Convergence factor with $\theta=1/2$}
We now focus on the remaining six algorithms NN$_{1\text{b}}$, NN$_{1\text{c}}$, NN$_{2\text{a}}$, NN$_{2\text{c}}$, NN$_{3\text{a}}$ and NN$_{3\text{c}}$.
Based on our analysis,
all six algorithms have shown the potentiel of being a good solver,
we thus compare them with a given relaxation parameter $\theta=1/2$ in two test cases.
Figure~\ref{fig:cv6} shows the behavior of the convergence factor as a function of the eigenvalues for the six algorithms.
\begin{figure}
\centering
\includegraphics[scale=0.25]{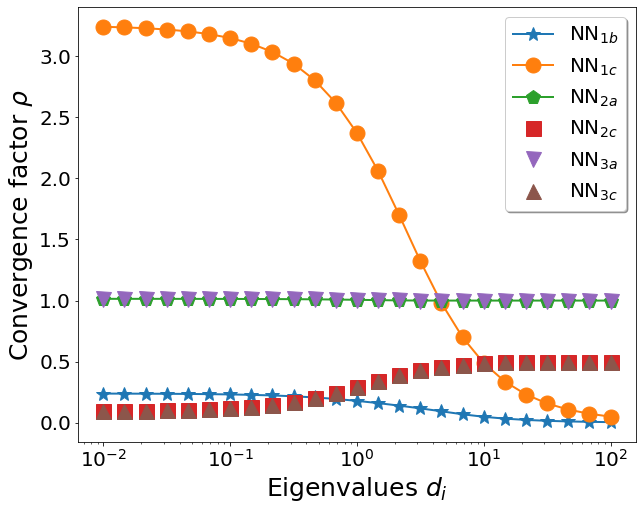}
\includegraphics[scale=0.25]{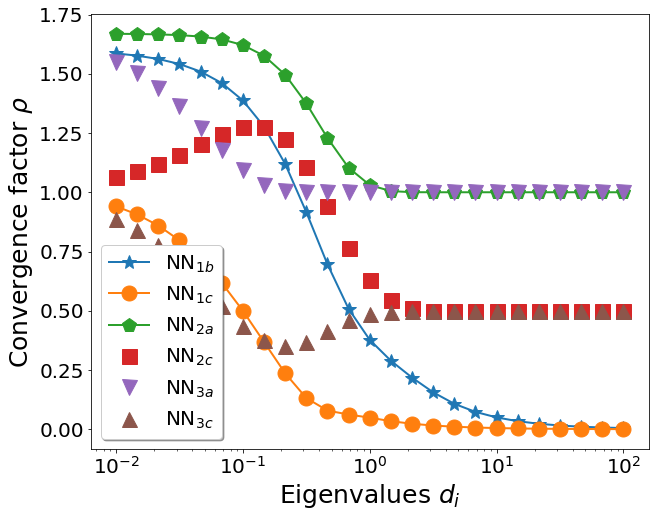}
\caption{Convergence factor with $\theta=1/2$ of the six algorithms as a function of the eigenvalues $d_i \in [10^{-2}, 10^2 ]$.
	Left: case A.
	Right: case B.}
\label{fig:cv6}
\end{figure}
In case A, 
we observe that NN$_{2\text{a}}$ and NN$_{3\text{a}}$ have identical behavior,
and similar for NN$_{2\text{c}}$ and NN$_{3\text{c}}$.
Indeed,
as explained in our analysis,
the convergence factors are the same in case A for NN$_{2\text{a}}$ and NN$_{3\text{a}}$,
and also for NN$_{2\text{c}}$ and NN$_{3\text{c}}$.
Furthermore,
NN$_{1\text{b}}$ and NN$_{1\text{c}}$ have similar behavior for large eigenvalues,
which has also been pointed out in our analysis.
And as expected, 
these two algorithms are good smoothers with $\theta=1/2$.
In particular,
NN$_{1\text{b}}$ outperforms the other five algorithms in case A,
that is both a good smoother and solver.
However,
this changes in case B.
More precisely, 
NN$_{2\text{a}}$ and NN$_{3\text{a}}$ have rather a symmetric behavior, 
as well as NN$_{2\text{c}}$ and NN$_{3\text{c}}$. 
And as shown in our analysis,
both NN$_{2\text{a}}$ and NN$_{3\text{a}}$ have the same behavior for large eigenvalues, 
and also NN$_{2\text{c}}$ and NN$_{3\text{c}}$.
Moreover,
NN$_{1\text{b}}$ and NN$_{1\text{c}}$ are both good smoothers,
and NN$_{1\text{c}}$ has a better performance than NN$_{1\text{b}}$ this time.

\subsection{Convergence factor with optimal $\theta$}
We then show the convergence behavior of each algorithm using their optimal relaxation parameter $\theta^{\star}$ determined by optimization.
Figure~\ref{fig:optcv6} shows the behavior of the convergence factor as a function of the eigenvalues for the six algorithms.
\begin{figure}
\centering
\includegraphics[scale=0.25]{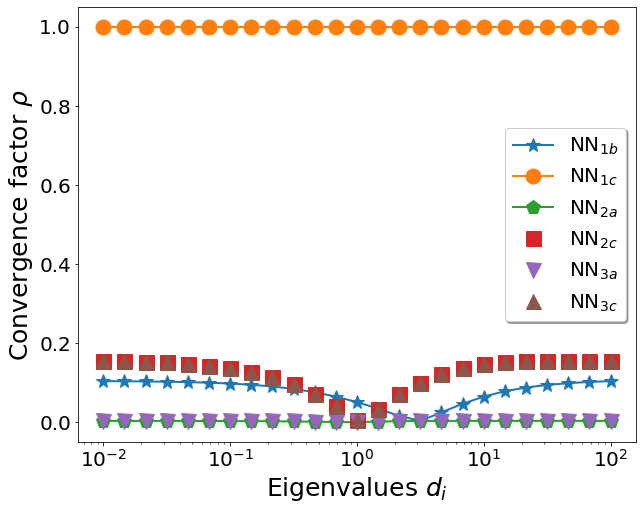}
\includegraphics[scale=0.25]{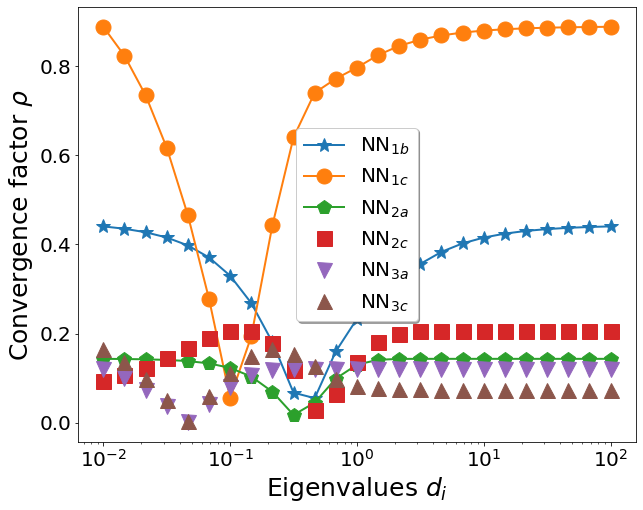}
\caption{Convergence factor with optimal relaxation parameter $\theta^{\star}$ of the six algorithms as a function of the eigenvalues $d_i \in [10^{-2}, 10^2 ]$.
	Left: case A.
	Right: case B.}
\label{fig:optcv6}
\end{figure}
In case A, 
NN$_{2\text{a}}$ and NN$_{3\text{a}}$ have once again identical behavior.
Indeed,
their convergence factors are the same in case A,
and both NN$_{2\text{a}}$ and NN$_{3\text{a}}$ have the same optimal relaxation parameter $\theta^{\star}_{\text{NN}_{2\text{a}}} =\theta^{\star}_{\text{NN}_{3\text{a}}}$, 
which corresponds to the theoretical value $\theta^*_{\text{NN}_{2\text{a}}}\approx 0.249$ as determined by~\eqref{eq:thetaNN2a}.
For the same reason, 
we observe the same behavior for NN$_{2\text{c}}$ and NN$_{3\text{c}}$,
where the optimal relaxation parameter $\theta^{\star}_{\text{NN}_{2\text{c}}} =\theta^{\star}_{\text{NN}_{3\text{c}}}=\theta^*_{\text{NN}_{2\text{c}}}\approx 0.385$ as determined by~\eqref{eq:thetaNN2c}.
As for NN$_{1\text{b}}$,
we find that the optimal relaxation parameter $\theta^{\star}_{\text{NN}_{1\text{b}}} =\theta^*_{\text{NN}_{1\text{b}}}\approx 0.446$ as determined by~\eqref{eq:thetaNN1b}.
However,
the optimal relaxation parameter for NN$_{1\text{c}}$ is $\theta^{\star}_{\text{NN}_{1\text{c}}}\approx0$, 
which cannot be determined by~\eqref{eq:thetaNN1c}.
As explained in our analysis,
the term $E_i-\nu^{-1}F_i$ in~\eqref{eq:rhoNN1c} is negative in case A,
thus the best option is to choose $\theta=0$ which becomes then a Schwarz type algorithm without overlap. 
In general,
all algorithms except NN$_{1\text{c}}$ have very good performance in case A,
and both NN$_{2\text{a}}$ and NN$_{3\text{a}}$ outperform the others with a
convergence factor around 10$^{-3}$.
Once again,
the behavior of the six algorithms becomes much different in case B.
While NN$_{1\text{c}}$ diverges in case A,
it converges in the test case B with the optimal relaxation parameter $\theta^{\star}_{\text{NN}_{1\text{c}}}=\theta^*_{\text{NN}_{1\text{c}}}\approx0.944$ as determined by~\eqref{eq:thetaNN1c}.
NN$_{1\text{b}}$ rather keeps a similar performance with the optimal relaxation parameter $\theta^{\star}_{\text{NN}_{1\text{b}}} =\theta^*_{\text{NN}_{1\text{b}}}\approx 0.278$ as determined by~\eqref{eq:thetaNN1b}.
NN$_{2\text{a}}$ and NN$_{3\text{a}}$ also have the same optimal relaxation parameter $\theta^{\star}_{\text{NN}_{2\text{a}}} =\theta^{\star}_{\text{NN}_{3\text{a}}}=\theta^*_{\text{NN}_{2\text{a}}}\approx 0.214$ as determined by~\eqref{eq:thetaNN2a}.
However,
for NN$_{2\text{c}}$ and NN$_{3\text{c}}$,
the optimal relaxation parameter of $\theta^{\star}_{\text{NN}_{2\text{c}}}\approx0.265$ is rather different from $\theta^{\star}_{\text{NN}_{3\text{c}}}\approx0.307$,
and both are different from the value determined by~\eqref{eq:thetaNN2c} using equioscillation $\theta^*_{\text{NN}_{2\text{c}}}\approx0.285$.
Indeed,
NN$_{2\text{c}}$ rather equioscillates the convergence value between large eigenvalues with some eigenvalue in the interval $[0.1,1]$,
whereas NN$_{3\text{c}}$ equioscillates the convergence value between small eigenvalues with some eigenvalue in the interval $[0.1,1]$.
In general,
all six algorithms converge in case B, 
NN$_{2\text{a}}$ and NN$_{3\text{a}}$ still outperform the others with NN$_{3\text{a}}$ slightly better than NN$_{2\text{a}}$.

\section{Conclusion}\label{sec:conclusion}
We introduced and analyzed nine new time domain decomposition methods based on Neumann-Neumann techniques for parabolic optimal control problems.
Our analysis shows that the Neumann correction step and the update step should be well-adjusted to the Dirichlet step to avoid potential divergence.
Moreover,
while it seems natural at first glance to preserve the forward-backward structure in the time subdomains as well, 
there are better choices that lead to substantially faster algorithms which can still be identified to be of forward-backward structure using changes of variables.
We also found many interesting mathematical connections between these algorithms, 
for instance the algorithms in Categories II and III have rather similar convergence behavior.
In terms of the performance,
NN$_{2\text{b}}$ and NN$_{3\text{b}}$ are bad algorithms,
the most natural algorithm NN$_{1\text{a}}$ is rather a good smoother,
and NN$_{2\text{a}}$ and NN$_{3\text{a}}$ with optimized relaxation parameter are much faster than the other algorithms and can be considered as highly efficient solvers.

Our study was restricted to the two subdomain case, 
but the algorithms can all naturally be written for many subdomains, 
and then one can also run them in parallel. 
They can also be used for more general parabolic constraints than the heat equation. 
Extensive numerical results will appear elsewhere.

\bibliographystyle{plain}
\bibliography{references}

\begin{thebibliography}{10}

\bibitem{Abbeloos2011}
Dirk Abbeloos, Moritz Diehl, Michael Hinze, and Stefan Vandewalle.
\newblock Nested multigrid methods for time-periodic, parabolic optimal control
  problems.
\newblock {\em Computing and Visualization in Science}, 14:27--38, 2011.

\bibitem{Alla2015}
Alessandro Alla and Stefan Volkwein.
\newblock Asymptotic stability of {POD} based model predictive control for a
  semilinear parabolic {PDE}.
\newblock {\em Advances in Computational Mathematics}, 41:1073--1102, 2015.

\bibitem{Bjorstad1986}
Petter~E. Bj\o{}rstad and Olof~B. Widlund.
\newblock Iterative methods for the solution of elliptic problems on regions
  partitioned into substructures.
\newblock {\em SIAM Journal on Numerical Analysis}, 23(6):1097--1120, 1986.

\bibitem{Borzi2011}
A.~Borzì and V.~Schulz.
\newblock {\em Computational Optimization of Systems Governed by Partial
  Differential Equations}.
\newblock Society for Industrial and Applied Mathematics, 2011.

\bibitem{Bunger2020}
A.~B\"{u}nger, S.~Dolgov, and M.~Stoll.
\newblock A low-rank tensor method for {PDE}-constrained optimization with
  isogeometric analysis.
\newblock {\em SIAM Journal on Scientific Computing}, 42(1):A140--A161, 2020.

\bibitem{Emmett2012}
M.~Emmett and M.~Minion.
\newblock {Toward an efficient parallel in time method for partial differential
  equations}.
\newblock {\em Communications in Applied Mathematics and Computational
  Science}, 7(1):105 -- 132, 2012.

\bibitem{Falgout2014}
R.~D. Falgout, S.~Friedhoff, Tz.~V. Kolev, S.~P. MacLachlan, and J.~B.
  Schroder.
\newblock Parallel time integration with multigrid.
\newblock {\em SIAM Journal on Scientific Computing}, 36(6):C635--C661, 2014.

\bibitem{Fang2022}
Liang Fang, Stefan Vandewalle, and Johan Meyers.
\newblock A parallel-in-time multiple shooting algorithm for large-scale
  {PDE}-constrained optimal control problems.
\newblock {\em Journal of Computational Physics}, 452:110926, 2022.

\bibitem{Farhat2003}
Charbel Farhat and Marion Chandesris.
\newblock Time-decomposed parallel time-integrators: theory and feasibility
  studies for fluid, structure, and fluid–structure applications.
\newblock {\em International Journal for Numerical Methods in Engineering},
  58(9):1397--1434, 2003.

\bibitem{Farhat1991}
Charbel Farhat and Francois-Xavier Roux.
\newblock A method of finite element tearing and interconnecting and its
  parallel solution algorithm.
\newblock {\em International Journal for Numerical Methods in Engineering},
  32(6):1205--1227, 1991.

\bibitem{Gander2015}
M.~J. Gander.
\newblock 50 years of time parallel time integration.
\newblock In T.~Carraro, M.~Geiger, S.~Körkel, and R.~Rannacher, editors, {\em
  Multiple Shooting and Time Domain Decomposition Methods}, pages 69--114.
  Springer, Heidelberg, 2015.

\bibitem{Gander2016}
Martin~J. Gander and Felix Kwok.
\newblock Schwarz methods for the time-parallel solution of parabolic control
  problems.
\newblock In Thomas Dickopf, Martin~J. Gander, Laurence Halpern, Rolf Krause,
  and Luca~F. Pavarino, editors, {\em Domain Decomposition Methods in Science
  and Engineering XXII}, pages 207--216, Cham, 2016. Springer International
  Publishing.

\bibitem{Gander2020}
Martin~J. Gander, Felix Kwok, and Julien Salomon.
\newblock Paraopt: A parareal algorithm for optimality systems.
\newblock {\em SIAM Journal on Scientific Computing}, 42(5):A2773--A2802, 2020.

\bibitem{Gander2023}
Martin~J. Gander and Liu-Di Lu.
\newblock New time domain decomposition methods for parabolic optimal control
  problems {I}: {D}irichlet-{N}eumann and {N}eumann-{D}irichlet algorithms.
\newblock Accepted with minor revision in SIAM Journal on Numerical Analysis,
  2023.

\bibitem{Gotschel2019}
Sebastian G\"{o}tschel and Michael~L. Minion.
\newblock An efficient parallel-in-time method for optimization with parabolic
  {PDE}s.
\newblock {\em SIAM Journal on Scientific Computing}, 41(6):C603--C626, 2019.

\bibitem{Gunzburger2011}
Max~D. Gunzburger and Angela Kunoth.
\newblock Space-time adaptive wavelet methods for optimal control problems
  constrained by parabolic evolution equations.
\newblock {\em SIAM Journal on Control and Optimization}, 49(3):1150--1170,
  2011.

\bibitem{Hackbusch1981}
W.~Hackbusch.
\newblock Numerical solution of linear and nonlinear parabolic control
  problems.
\newblock In Alfred Auslender, Werner Oettli, and Josef Stoer, editors, {\em
  Optimization and Optimal Control}, pages 179--185, Berlin, Heidelberg, 1981.
  Springer Berlin Heidelberg.

\bibitem{Halpern2010}
Laurence Halpern and Jérémy Szeftel.
\newblock Optimized and quasi-optimal {S}chwarz waveform relaxation for the
  one-dimensional {S}chrödinger equation.
\newblock {\em Mathematical Models and Methods in Applied Sciences},
  20(12):2167--2199, 2010.

\bibitem{Heinkenschloss2005}
Matthias Heinkenschloss.
\newblock A time-domain decomposition iterative method for the solution of
  distributed linear quadratic optimal control problems.
\newblock {\em Journal of Computational and Applied Mathematics},
  173(1):169--198, 2005.

\bibitem{Hinze2009}
M.~Hinze, R.~Pinnau, M.~Ulbrich, and S.~Ulbrich.
\newblock {\em Optimization with {PDE} Constraints}.
\newblock Springer Dordrecht, 2009.

\bibitem{Iapichino2016}
Laura Iapichino, Stefan Trenz, and Stefan Volkwein.
\newblock Reduced-order multiobjective optimal control of semilinear parabolic
  problems.
\newblock In B{\"u}lent Karas{\"o}zen, Murat Manguo{\u{g}}lu, M{\"u}nevver
  Tezer-Sezgin, Serdar G{\"o}ktepe, and {\"O}m{\"u}r U{\u{g}}ur, editors, {\em
  Numerical Mathematics and Advanced Applications ENUMATH 2015}, pages
  389--397, Cham, 2016. Springer International Publishing.

\bibitem{Kammann2013}
Eileen Kammann, Fredi Tröltzsch, and Stefan Volkwein.
\newblock A posteriori error estimation for semilinear parabolic optimal
  control problems with application to model reduction by {POD}.
\newblock {\em ESAIM: Mathematical Modelling and Numerical Analysis},
  47(2):555–581, 2013.

\bibitem{Kollmann2013469}
M.~Kollmann, M.~Kolmbauer, U.~Langer, M.~Wolfmayr, and W.~Zulehner.
\newblock A robust finite element solver for a multiharmonic parabolic optimal
  control problem.
\newblock {\em Computers \& Mathematics with Applications}, 65(3):469--486,
  2013.
\newblock Efficient Numerical Methods for Scientific Applications.

\bibitem{Kunisch2004}
K.~Kunisch, S.~Volkwein, and L.~Xie.
\newblock {HJB}-{POD}-based feedback design for the optimal control of
  evolution problems.
\newblock {\em SIAM Journal on Applied Dynamical Systems}, 3(4):701--722, 2004.

\bibitem{Kwok2017}
Felix Kwok.
\newblock On the time-domain decomposition of parabolic optimal control
  problems.
\newblock In Chang-Ock Lee, Xiao-Chuan Cai, David~E. Keyes, Hyea~Hyun Kim, Axel
  Klawonn, Eun-Jae Park, and Olof~B. Widlund, editors, {\em Domain
  Decomposition Methods in Science and Engineering XXIII}, pages 55--67, Cham,
  2017. Springer International Publishing.

\bibitem{Lelarasmee1982}
E.~Lelarasmee, A.~E. Ruehli, and A.~L. Sangiovanni-Vincentelli.
\newblock The waveform relaxation method for time-domain analysis of large
  scale integrated circuits.
\newblock {\em IEEE Transactions on Computer-Aided Design of Integrated
  Circuits and Systems}, 1(3):131--145, 1982.

\bibitem{Li2017358}
Buyang Li, Jun Liu, and Mingqing Xiao.
\newblock A new multigrid method for unconstrained parabolic optimal control
  problems.
\newblock {\em Journal of Computational and Applied Mathematics}, 326:358--373,
  2017.

\bibitem{Lions1971}
J.-L. Lions.
\newblock {\em Optimal Control of Systems Governed by Partial Differential
  Equations}.
\newblock 170. Springer-Verlag Berlin Heidelberg, 1 edition, 1971.

\bibitem{Lions2002}
Jacques-Louis Lions, Yvon Maday, and Gabriel Turinici.
\newblock A parareal in time procedure for the control of partial differential
  equations.
\newblock {\em Comptes Rendus Mathematique}, 335(4):387--392, 2002.

\bibitem{Troltzsch2010}
Fredi Tröltzsch.
\newblock {\em Optimal Control of Partial Differential Equations: Theory,
  Methods and Applications}, volume 112.
\newblock Graduate Studies in Mathematics, 2010.

\bibitem{Langer2016}
Sergey~Repin Ulrich~Langer and Monika Wolfmayr.
\newblock Functional a posteriori error estimates for time-periodic parabolic
  optimal control problems.
\newblock {\em Numerical Functional Analysis and Optimization},
  37(10):1267--1294, 2016.

\end{thebibliography}

\appendix

\section{Pair transmission conditions}\label{sec:app1}

Let us consider a modified algorithm NN$_{2\text{a}}$,
that is,
we first solve the Dirichlet step
\[
\begin{aligned}
  &\left\{
    \begin{aligned}
      \begin{pmatrix}
        \dot z_{1,i}^k\\
        \dot \mu_{1,i}^k
      \end{pmatrix}
      +
      \begin{pmatrix}
        d_i & -\nu^{-1}\\
        -1 & -d_i
      \end{pmatrix}
      \begin{pmatrix}
        z_{1,i}^k\\
        \mu_{1,i}^k
      \end{pmatrix}
      &=
      \begin{pmatrix}
        0\\
        0
      \end{pmatrix} \text{ in } \Omega_1,\\
      z_{1,i}^k(0) &= 0,\\
      z_{1,i}^k(\alpha) & = f_{\alpha,i}^{k-1},
    \end{aligned}
  \right.\\
  &\left\{
    \begin{aligned}
      \begin{pmatrix}
        \dot z_{2,i}^k\\
        \dot \mu_{2,i}^k
      \end{pmatrix}
      +
      \begin{pmatrix}
        d_i & -\nu^{-1}\\
        -1 & -d_i
      \end{pmatrix}
      \begin{pmatrix}
        z_{2,i}^k\\
        \mu_{2,i}^k
      \end{pmatrix}
      &=
      \begin{pmatrix}
        0\\
        0
      \end{pmatrix} \text{ in } \Omega_2,\\
      z_{2,i}^k(\alpha) &= g_{\alpha,i}^{k-1},\\
      \mu_{2,i}^k(T) + \gamma z_{2,i}^k(T) &= 0,
    \end{aligned}
  \right.
   \end{aligned}
\]
and then correct the result by the Neumann step
\[\begin{aligned}
  &\left\{
    \begin{aligned}
      \begin{pmatrix}
        \dot \psi_{1,i}^k\\
        \dot \phi_{1,i}^k
      \end{pmatrix}
      +
      \begin{pmatrix}
        d_i & -\nu^{-1} \\
        -1 & -d_i
      \end{pmatrix}
      \begin{pmatrix}
        \psi_{1,i}^k\\
        \phi_{1,i}^k
      \end{pmatrix}
      &=
      \begin{pmatrix}
        0\\
        0
      \end{pmatrix} \text{ in } \Omega_1,\\
      \psi_{1,i}^k(0) &= 0,\\
      \dot\psi_{1,i}^k(\alpha) & = \dot z_{1,i}^k(\alpha) - \dot z_{2,i}^k(\alpha),
    \end{aligned}
  \right.\\
  &\left\{
    \begin{aligned}
      \begin{pmatrix}
        \dot \psi_{2,i}^k\\
        \dot \phi_{2,i}^k
      \end{pmatrix}
      +
      \begin{pmatrix}
        d_i & -\nu^{-1}\\
        -1 & -d_i
      \end{pmatrix}
      \begin{pmatrix}
        \psi_{2,i}^k\\
        \phi_{2,i}^k
      \end{pmatrix}
      &=
      \begin{pmatrix}
        0\\
        0
      \end{pmatrix} \text{ in } \Omega_2,\\
      \dot\psi_{2,i}^k(\alpha) &= \dot z_{2,i}^k(\alpha) - \dot z_{1,i}^k(\alpha),\\
      \phi_{2,i}^k(T) + \gamma \psi_{2,i}^k(T) &= 0.
    \end{aligned}
  \right.
\end{aligned}\]
and update the transmission condition by
\[
  f_{\alpha,i}^k:=f_{\alpha,i}^{k-1} - \theta_1 \big(\psi_{1,i}^k(\alpha) + \psi_{2,i}^k(\alpha)\big), \quad g_{\alpha,i}^k:=g_{\alpha,i}^{k-1} - \theta_2 \big(\psi_{1,i}^k(\alpha) + \psi_{2,i}^k(\alpha)\big),
\]
with $\theta_1,\theta_2>0$. 
Following the same analysis as in Section~\ref{sec:NN2a},
we find,
\[
  \begin{pmatrix}
    f_{\alpha,i}^k\\
    g_{\alpha,i}^k
  \end{pmatrix} = \begin{pmatrix}
    1-\theta_1 E_i & -\theta_1 F_i\\
    -\theta_2E_i & 1-\theta_2F_i
  \end{pmatrix}\begin{pmatrix}
    f_{\alpha,i}^{k-1}\\
    g_{\alpha,i}^{k-1}
  \end{pmatrix}.
\]
In particular, the eigenvalues of the iteration matrix are 1 and $1-(\theta_1 E_i + \theta_2 F_i )$. 
Thus,
the modified algorithm NN$_{2\text{a}}$ does not converge in this form.
This divergence still stays even by considering the update step~\eqref{eq:NN1atran} for the pair transmission conditions.
More generally, 
we have the same behavior for NN$_{2\text{b}}$, NN$_{2\text{c}}$, NN$_{3\text{a}}$, NN$_{3\text{b}}$ and NN$_{3\text{c}}$,
if we keep a pair of transmission conditions $(f_{\alpha,i}^k, g_{\alpha,i}^k)$.
\end{document}